\documentclass[11pt,reqno]{amsart}

\usepackage[a4paper,margin=1.2in]{geometry}
\usepackage{setspace}
\usepackage{parskip}

\usepackage{amsmath, amssymb, amsthm, amsfonts}
\usepackage{mathtools}
\usepackage{mathrsfs}
\usepackage{graphicx}
\usepackage{array,multirow}
\usepackage{enumitem}
\usepackage[dvipsnames]{xcolor}
\usepackage{tikz}
\usetikzlibrary{calc,decorations.pathmorphing}
\usepackage{tikz-cd}

\tikzset{curve/.style={settings={#1},to path={(\tikztostart)
			.. controls ($(\tikztostart)!\pv{pos}!(\tikztotarget)!\pv{height}!270:(\tikztotarget)$)
			and ($(\tikztostart)!1-\pv{pos}!(\tikztotarget)!\pv{height}!270:(\tikztotarget)$)
			.. (\tikztotarget)\tikztonodes}},
	settings/.code={\tikzset{quiver/.cd,#1}
		\def\pv##1{\pgfkeysvalueof{/tikz/quiver/##1}}},
	quiver/.cd,pos/.initial=0.35,height/.initial=0}

\tikzset{tail reversed/.code={\pgfsetarrowsstart{tikzcd to}}}
\tikzset{2tail/.code={\pgfsetarrowsstart{Implies[reversed]}}}
\tikzset{2tail reversed/.code={\pgfsetarrowsstart{Implies}}}

\tikzset{no body/.style={/tikz/dash pattern=on 0 off 1mm}}

\usepackage{xurl}
\usepackage[
pdfusetitle,
colorlinks=true,
linkcolor={Brown},
citecolor={Brown},
urlcolor={Brown},
linktoc=page]{hyperref}

\newtheorem{thm}{Theorem}[section]
\newtheorem{lem}[thm]{Lemma}
\newtheorem{prop}[thm]{Proposition}
\newtheorem{cor}[thm]{Corollary}

\theoremstyle{definition}

\theoremstyle{remark}
\newtheorem{rem}[thm]{Remark}
\newtheorem{rec}[thm]{Recollection}
\newtheorem{nota}[thm]{Notation}
\newtheorem*{Ack}{Acknowledgements}

\numberwithin{equation}{section}
\newcommand{\N}{\ensuremath{\mathbb{N}}}
\newcommand{\defeq}{\vcentcolon=}
\newcommand{\inj}[1]{%
	\begin{tikzcd}[cramped, sep=small, ampersand replacement=\&]
		{} \arrow[hook, r, "#1"] \& {}
	\end{tikzcd}%
}
\newcommand{\sur}[1]{%
	\begin{tikzcd}[cramped, sep=small, ampersand replacement=\&]
			{} \arrow[two heads, r, "#1"] \& {}
	\end{tikzcd}%
}

\DeclareMathOperator{\mmod}{\mathsf{mod}}
\DeclareMathOperator{\soc}{soc}
\DeclareMathOperator{\rad}{rad}
\DeclareMathOperator{\hd}{hd}
\DeclareMathOperator{\J}{J}
\DeclareMathOperator{\Hom}{Hom}
\DeclareMathOperator{\im}{im}
\DeclareMathOperator{\coker}{coker}
\DeclareMathOperator{\vspan}{span}
\DeclareMathOperator{\Char}{char}
\DeclareMathOperator{\ppdim}{ppdim}

\title{The permutation dimension of the Klein 4-group}
\date{}
\author{Henry Harman}
\thanks{At the time this work was completed, the author was affiliated with the University of Warwick.}
\address{Mathematics Department, Imperial College London, London, SW7 2AZ, UK}
\email{h.harman25@imperial.ac.uk}
\hypersetup{pdfkeywords={modular representation theory, permutation modules, permutation dimension}}

\begin{document}
	
\begin{abstract}
	Let $G$ be a finite group and $k$ a field of characteristic $p > 0$. Balmer and Gallauer's recent result on finite $p$-permutation resolutions of $kG$-modules motivates the study of an intriguing new invariant; the $p$-permutation dimension. Following Walsh's success with cyclic groups of prime order, we compute the (global) $p$-permutation \linebreak dimension of the Klein 4-group in characteristic $p \defeq 2$, along with the dimensions for each of its indecomposable modules.
\end{abstract}



\maketitle

\vspace{-0.5\baselineskip}

{
	\setlength{\parskip}{0pt}
	\tableofcontents
}

\vspace{-1.5\baselineskip}

\section{Introduction}
\label{sec:intro}
\emph{Let $G$ be a finite group and $k$ a field of characteristic $p > 0$. All modules are assumed to be finitely generated left modules.}

A $kG$-module is called a \emph{permutation module} if it admits a $k$-basis closed under the action of $G$, and a \emph{$p$-permutation module} if it admits a $k$-basis closed under the action of a Sylow $p$-subgroup of $G$. Equivalently, $p$-permutation modules are those which occur as direct summands of permutation modules, much like projectives can be characterised as direct summands of free modules. Every free module is a permutation module, so $p$-permutation modules simultaneously generalise both permutation and projective modules.

Resolutions by projective modules are important objects of study in many contexts, but the projective dimension is just far too coarse when it comes to group algebras; any non-projective $kG$-module has infinite projective dimension.

One might hope that using $p$-permutation modules in place of projectives will provide a more meaningful distinction between $kG$-modules. Balmer and Gallauer's result indeed affirms this, but first we recall some key definitions:

Let $M$ be a $kG$-module. A \emph{$p$-permutation resolution of $M$} is an exact complex:
\[\begin{tikzcd}[sep=small]
	\cdots & {P_{i}} & \cdots & {P_{2}} & {P_{1}} & {P_{0}} & M & 0
	\arrow[from=1-1, to=1-2]
	\arrow[from=1-2, to=1-3]
	\arrow[from=1-3, to=1-4]
	\arrow[from=1-4, to=1-5]
	\arrow[from=1-5, to=1-6]
	\arrow[from=1-6, to=1-7]
	\arrow[from=1-7, to=1-8]
\end{tikzcd}\]
where the $P_{i}$ are $p$-permutation modules.
We define its \emph{length} as $\sup \left \{i \in \N \vcentcolon P_{i} \neq 0 \right \}$.

\pagebreak
The minimal length of such a resolution is called the \emph{$p$-permutation dimension of $M$}, denoted $\ppdim(M)$. Balmer and Gallauer \cite{BG23} showed that every $kG$-module admits a finite resolution by $p$-permutation modules, hence $\ppdim(M) < \infty$ for any $M$.

We also define the \emph{global $p$-permutation dimension of $G$ over $k$} as follows: \[\ppdim_{k}(G) \defeq \sup \left\{\ppdim(M) \vcentcolon M \in \mmod(kG) \right\}\]

As always, the interesting case occurs when $p$ divides $|G|$. Otherwise, $kG$ is semisimple and every module is projective, hence $p$-permutation, meaning $\ppdim_{k}(G) = 0$.

Due to the recency of this result, little is known about the homological implications of the $p$-permutation dimension. Walsh \cite{Wal25} took the first steps towards gaining some intuition, proving the following theorem:

\begin{thm}[Walsh, 2026]
	Let $k$ be a field of characteristic $p > 0$. Then:
	\[
	\ppdim_{k}\left(C_{p}\right) = p-2
	\]
\end{thm}

Note that when $G$ is a $p$-group, permutation modules and $p$-permutation modules are exactly the same thing. Moreover, the group algebra $kG$ is a local ring, so free and projective modules are also exactly the same thing. When $G \defeq C_{p}$, the group algebra has finite representation type, further simplifying the picture.

In this article, we focus on the Klein 4-group $V_{4}$ over a field $k$ of characteristic two. The group algebra $kV_{4}$ has \emph{tame} representation type, meaning that despite there being infinitely many indecomposables up to isomorphism, there still exists a `reasonable' way of classifying them.

We recall this classification in Section \ref{sec:class} and fix notation for each indecomposable. Namely, the indecomposables consist of the regular and trivial modules, $kV_{4}$ and $k$, in addition to three infinite families denoted by $M_{2n+1}$\,, $W_{2n+1}$ and $E_{f,n}$\,, with $n \geq 1$ and either $f \in k[t]$ a monic irreducible polynomial or $f = \infty$.\footnote{We also set $E_{f} \defeq E_{f,1}$.} We are now ready to state the main result:

\begin{thm}
\label{thm:main}
	Let $k$ be a field of characteristic $p \defeq 2$. Then:
	\[
	\ppdim_{k}\left(V_{4}\right) = 2
	\]
	
	Moreover, for $N$ an indecomposable $kV_{4}$-module, we have:
	\[
	\ppdim(N) = 
	\begin{cases}
		0 & \text{if $N \in \big\{k \,, \; E_{t} \,, \; E_{t+1_{k}} \,, \; E_{\infty} \,, \; kV_{4}\big\}$}\\
		1 & \text{if $N \in \big\{M_{3} \,, \; M_{5} \,, \; M_{7} \,, \; W_{3} \,, \; W_{5} \,, \; E_{t,2} \,, \; E_{(t+1_{k}),2} \,, \; E_{\infty,2}\big\}$}\\
		2 & \text{otherwise}
	\end{cases}
	\]
\end{thm}

Resolutions are given in Section \ref{sec:resol}. Constructing them is relatively easy, but proving that they are minimal is more of a challenge. We establish their minimality in Section \ref{sec:minim}.

\section{Indecomposable Modules}
\label{sec:indec}
For the remainder of this article, we fix a field $k$ of characteristic $p \defeq 2$. In this section, we recall the classification of indecomposable $kV_{4}$-modules, taking note of their corresponding duals and Heller shifts. We also identify the indecomposable permutation modules.

\subsection{The Klein 4-Group}
\label{sec:klein}
\mbox{}\\
The Klein 4-group has the following group presentation:
\[
V_{4} \defeq \left \langle \sigma, \tau \: \middle | \: \sigma^{2} = \tau^{2} = (\sigma\tau)^{2} = e \right \rangle
\]
Hence $V_{4} = \left\{e, \sigma, \tau, \sigma\tau\right\}$, which we can take as a $k$-basis for $kV_{4}$. Since $\Char(k) = 2$, we have $1_{k} + 1_{k} = 0_{k}$ and $-1_{k} = 1_{k}$, which we will make frequent and implicit use of. The group algebra $kV_{4}$ can be identified with $k[a,b]\,\big/\langle a^{2}, b^{2} \rangle$, which is generated as a $k$-algebra by $\{a,b\}$ along with the relations $a^{2} = b^{2} = ab - ba = 0_{k}$, or equivalently, $a^{2} = b^{2} = (a+b)^{2} = 0_{k}$. The following isomorphism of $k$-algebras allows us to translate between $k[a,b]\,\big/\langle a^{2}, b^{2} \rangle$-modules and $kV_{4}$-modules:
\[
\phi \vcentcolon k[a,b]\,\big/\langle a^{2}, b^{2} \rangle \to kV_{4}, \hspace{2mm} a \mapsto \sigma + e, \hspace{2mm} b \mapsto \tau + e 
\]
\begin{rem}
	Under this identification, we record some key features of the $k$-algebra $kV_{4}$:
	\begin{enumerate}[label=(\roman*)]
		
		\item $kV_{4}$ is a commutative Artinian local ring with unique maximal ideal $\J(kV_{4}) = \langle a, b \rangle$ and unique minimal ideal $\soc(kV_{4}) = \langle ab \rangle$. A natural $k$-basis is $\{1_{k}, a, b, ab\}$.
		
		\item Projective, injective and free modules are exactly the same thing.
		
		\item Permutation and $p$-permutation modules are exactly the same thing.\footnote{We consequently speak of `permutation resolutions' and `permutation dimensions' from now on.}
		
		\item The trivial module $k$ is the unique simple module and the regular module $kV_{4}$ is the unique projective indecomposable module (up to isomorphism).
		
	\end{enumerate}
\end{rem}

\subsection{The Classification of Indecomposable Modules}
\label{sec:class}
\mbox{}\\
We now recall the classification of indecomposable $kV_4$-modules, which
goes back to (most notably) the work of Ba\v{s}ev~\cite{Bas61} as well as Heller and Reiner~\cite{HR61}. Modern accounts can be found in \cite[Thm.~4.3.3]{Ben91}, \cite[p.~233]{Web16} or
\cite[Thm.~2.1]{CM12}.

Observe that specifying a $kV_{4}$-module amounts to specifying the action of $a$ and $b$ on a $k$-basis, such that the relations $a^{2} = b^{2} = (a+b)^{2} = 0_{k}$ are respected. This information can be presented in the form of a diagram. For example, the regular and trivial modules $kV_{4}$ and $k$ are given as follows:
%
%
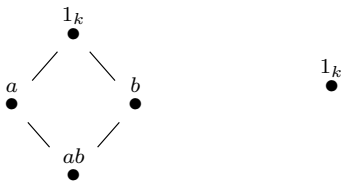
\[\begin{tikzcd}[sep=tiny]
	& {\overset{1_{k}}{\bullet}} \\
	{\overset{a}{\bullet}} && {\overset{b}{\bullet}} \\
	& {\overset{ab}{\bullet}}
	\arrow[no head, from=1-2, to=2-1]
	\arrow[no head, from=1-2, to=2-3]
	\arrow[no head, from=2-1, to=3-2]
	\arrow[no head, from=2-3, to=3-2]
\end{tikzcd}
\hspace{20mm}
\begin{tikzcd}[sep=tiny]
	{\overset{1_{k}}\bullet}
\end{tikzcd}\]
In general, the indecomposable $kV_{4}$-modules are exhibited using a collection of nodes indexed by a $k$-basis, together with diagonal lines between the nodes. We interpret these lines as being oriented from the higher node to the lower node.

A south-west/south-east branch indicates that $a$/$b$ sends the basis element at the higher node to the basis element at the lower node. The lack of a south-west/south-east branch coming from a node indicates that $a$/$b$ annihilates the corresponding basis element.

The non-trivial odd-dimensional indecomposables come in two families; either they are `M-shaped' or `W-shaped'. For each $n\geq 1$, there are $kV_{4}$-modules $M_{2n+1}$ and $W_{2n+1}$:
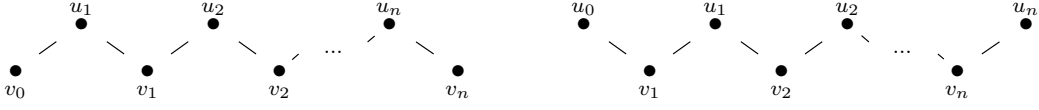
\[\begin{tikzcd}[sep=tiny]
	& {\overset{u_{1}}\bullet} && {\overset{u_{2}}\bullet} && {} & {\overset{u_{n}}\bullet} \\
	{\underset{v_{0}}\bullet} && {\underset{v_{1}}\bullet} && {\underset{v_{2}}\bullet} & {} && {\underset{v_{n}}\bullet}
	\arrow[no head, from=1-2, to=2-1]
	\arrow[no head, from=1-2, to=2-3]
	\arrow[no head, from=1-4, to=2-3]
	\arrow[no head, from=1-4, to=2-5]
	\arrow["\ldots"{description, pos=0.4}, draw=none, from=1-6, to=2-6]
	\arrow[shorten >=18pt, no head, from=1-7, to=2-6]
	\arrow[no head, from=1-7, to=2-8]
	\arrow[shorten >=14pt, no head, from=2-5, to=1-6]
\end{tikzcd}
\hspace{10mm}
\begin{tikzcd}[sep=tiny]
	{\overset{u_{0}}\bullet} && {\overset{u_{1}}\bullet} && {\overset{u_{2}}\bullet} & {} && {\overset{u_{n}}\bullet} \\
	& {\underset{v_{1}}\bullet} && {\underset{v_{2}}\bullet} && {} & {\underset{v_{n}}\bullet}
	\arrow[no head, from=1-1, to=2-2]
	\arrow[no head, from=1-3, to=2-2]
	\arrow[no head, from=1-3, to=2-4]
	\arrow[shorten >=18pt, no head, from=1-5, to=2-6]
	\arrow["\ldots"{description, pos=0.4}, draw=none, from=1-6, to=2-6]
	\arrow[no head, from=1-8, to=2-7]
	\arrow[no head, from=2-4, to=1-5]
	\arrow[shorten >=14pt, no head, from=2-7, to=1-6]
\end{tikzcd}\]
We have $\dim_{k}(M_{2n+1}) = \dim_{k}(W_{2n+1}) = 2n+1$. Our standard labelling for their $k$-bases is $\left\{u_{i} \vcentcolon 1 \leq i \leq n\right\} \cup \left\{v_{i} \vcentcolon 0 \leq i \leq n\right\}$ and $\left\{u_{i} \vcentcolon 0 \leq i \leq n\right\} \cup \left\{v_{i} \vcentcolon 1 \leq i \leq n\right\}$ respectively.

The non-projective even-dimensional indecomposables $E_{f,n}$ are parametrised by pairs $(f, n)$ where $n \geq 1$ and either $f \in k[t]$ is a monic irreducible polynomial or $f = \infty$.\footnote{The use of the symbol $\infty$ is motivated by the case in which $k$ is algebraically closed. The monic irreducible polynomials are then of the form $t + \alpha$ for $\alpha \in k$, so the $E_{f,n}$ modules are parametrised by points on the projective line $k \cup \{\infty\}$ paired with positive integers.}
\linebreak
In the first case, we have $f^{n} = t^{m} + \sum_{i = 0}^{m - 1}(\alpha_{i} t^{i})$ for some $\alpha_{i} \in k$, where $m \defeq n \deg(f)$. The diagram for $E_{f,n}$ is then given as follows:
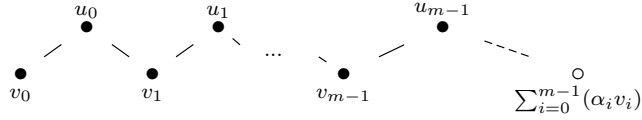
\[\begin{tikzcd}[sep=tiny]
	& {\overset{u_{0}}\bullet} && {\overset{u_{1}}\bullet} & {} && {\overset{u_{m - 1}}\bullet} \\
	{\underset{v_{0}}\bullet} && {\underset{v_{1}}\bullet} && {} & {\underset{v_{m - 1}}\bullet} && {\underset{\sum_{i = 0}^{m - 1}(\alpha_{i} v_{i})}{\circ}}
	\arrow[no head, from=1-2, to=2-1]
	\arrow[no head, from=1-2, to=2-3]
	\arrow[no head, from=1-4, to=2-3]
	\arrow[shorten >=18pt, no head, from=1-4, to=2-5]
	\arrow["\ldots"{description, pos=0.4}, draw=none, from=1-5, to=2-5]
	\arrow[no head, from=1-7, to=2-6]
	\arrow[dashed, no head, from=1-7, to=2-8]
	\arrow[shorten >=18pt, no head, from=2-6, to=1-5]
\end{tikzcd}\]
The action of $b$ on $u_{m - 1}$ is more complicated than what we have previously encountered; $u_{m - 1}$ is sent to the $k$-linear combination $\sum_{i = 0}^{m - 1}(\alpha_{i} v_{i})$. This is indicated in the diagram using a dashed south-east branch coming from $u_{m - 1}$ and an `empty' node to signify that $\sum_{i = 0}^{m - 1}(\alpha_{i} v_{i})$ is not a basis element. We have $\dim_{k}(E_{f,n}) = 2m$ and our standard labelling for the $k$-basis is $\left\{u_{i} \vcentcolon 0 \leq i \leq m - 1\right\} \cup \left\{v_{i} \vcentcolon 0 \leq i \leq m - 1\right\}$. When $f=t$, each $\alpha_{i}$ is equal to $0_{k}$ and $m=n$, so the diagram for $E_{t,n}$ takes a particularly nice form:
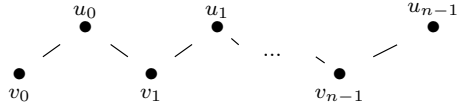
\[\begin{tikzcd}[sep=tiny]
	& {\overset{u_{0}}\bullet} && {\overset{u_{1}}\bullet} & {} && {\overset{u_{n - 1}}\bullet} \\
	{\underset{v_{0}}\bullet} && {\underset{v_{1}}\bullet} && {} & {\underset{v_{n - 1}}\bullet}
	\arrow[no head, from=1-2, to=2-1]
	\arrow[no head, from=1-2, to=2-3]
	\arrow[no head, from=1-4, to=2-3]
	\arrow[shorten >=18pt, no head, from=1-4, to=2-5]
	\arrow["\ldots"{description, pos=0.4}, draw=none, from=1-5, to=2-5]
	\arrow[no head, from=1-7, to=2-6]
	\arrow[shorten >=18pt, no head, from=2-6, to=1-5]
\end{tikzcd}\]
The case $f = \infty$ describes a `reflected version' of $E_{t,n}$. The diagram for $E_{\infty,n}$ is given by:
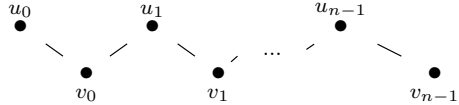
\[\begin{tikzcd}[sep=tiny]
	{\overset{u_{0}}\bullet} && {\overset{u_{1}}\bullet} && {} & {\overset{u_{n - 1}}\bullet} \\
	& {\underset{v_{0}}\bullet} && {\underset{v_{1}}\bullet} & {} && {\underset{v_{n - 1}}\bullet}
	\arrow[no head, from=1-1, to=2-2]
	\arrow[no head, from=1-3, to=2-2]
	\arrow[no head, from=1-3, to=2-4]
	\arrow["\ldots"{description, pos=0.4}, draw=none, from=1-5, to=2-5]
	\arrow[shorten >=13pt, no head, from=1-6, to=2-5]
	\arrow[no head, from=1-6, to=2-7]
	\arrow[shorten >=14pt, no head, from=2-4, to=1-5]
\end{tikzcd}\]
We therefore have $\dim_{k}(E_{\infty,n}) = 2n$ with the following standard labelling for the $k$-basis: $\left\{u_{i} \vcentcolon 0 \leq i \leq n - 1\right\} \cup \left\{v_{i} \vcentcolon 0 \leq i \leq n - 1\right\}$. When $n=1$, we set $E_{f} \defeq E_{f,1}$.

\begin{thm}[Ba\v{s}ev; Heller and Reiner]
\label{thm:class}
	The following is a complete and irredundant list of representatives for the isomorphism classes of indecomposable $kV_{4}$-modules:
	
	$kV_{4}$, $k$, $M_{2n+1}$, $W_{2n+1}$, $E_{f,n}$
	
	where $n \geq 1$ and either $f \in k[t]$ is a monic irreducible polynomial or $f = \infty$.
\end{thm}

\begin{proof}
	See \cite{Bas61,HR61}. For modern accounts, see
	\cite[\S 4.3]{Ben91} or \cite{CM12}.
\end{proof}

\subsection{Duality}
\label{sec:dual}
\begin{rec}
\label{rec:dual}
	To each $kV_{4}$-module $M$ we associate a \emph{dual module} $M^{*} \defeq \Hom_{k}(M, k)$, where $gf \vcentcolon m \mapsto f(g^{-1}m)$ for $g \in V_{4}$ and $f \in M^{*}$. Moreover, to each morphism of \linebreak $kV_{4}$-modules $\phi \in \Hom_{kV_{4}}(M,N)$, we associate a \emph{dual morphism} $\phi^{*} \in \Hom_{kV_{4}}(N^{*},M^{*})$ where $\phi^{*} \vcentcolon f \mapsto f \circ \phi$. We record some basic facts about duality:
	\begin{enumerate}[label=(\roman*)]
		
		\item $M \cong M^{**}$.
		
		\item $\phi \text{ is surjective} \implies \phi^{*} \text{ is injective}$,\, $\phi \text{ is injective} \implies \phi^{*} \text{ is surjective}$.
		
		\item Permutation modules are self-dual.
		
		\item $M \text{ is indecomposable} \iff M^{*} \text{ is indecomposable}$.
		
	\end{enumerate}
\end{rec}

The following can be found in \cite[Lem. 2.6, Prop. 3.3]{CM12}.

\begin{lem}
\label{lem:dual}
	The even-dimensional indecomposables are self-dual. In the odd-dimensional case, we have $M_{2n+1}^{*} \cong W_{2n+1}$, $W_{2n+1}^{*} \cong M_{2n+1}$ and $k^{*} \cong k$ for each $n \geq 1$.
\end{lem}

\begin{rem}
	The $p$-permutation dimension is not invariant under duality. Indeed, we shall see that $M_{7}^{*}\cong W_{7}$, while $\ppdim(M_{7})=1$ and $\ppdim(W_{7})=2$.
\end{rem}

\subsection{Heller Shifts}
\label{sec:hell}
\begin{rec}
	A surjection of $kV_{4}$-modules $\phi \vcentcolon M \sur{} N$ is called \emph{essential} if $\phi(M') \subsetneq N$ for all proper submodules $M' \subsetneq M$. A \emph{projective cover} of $M$ is an essential surjection $\phi \vcentcolon P \sur{} M$ for which $P$ is projective. We let $\Omega(M) \defeq \ker(\phi)$ and call this the \emph{Heller shift} of $M$; $\Omega(-)$ is known as the \emph{Heller operator}. The existence and uniqueness of projective covers ensures that the Heller operator is well-defined up to isomorphism. Any projective module certainly vanishes under $\Omega(-)$ and $\Omega\left(M \oplus N\right) \cong \Omega(M) \oplus \Omega(N)$.
\end{rec}

We record two useful results found in \cite[Thm. 7.3.1, Prop. 7.3.2]{Web16}.

\begin{prop}
\label{prop:nak}
	Let $M$ be a $kV_{4}$-module. The canonical quotient map $\pi \vcentcolon M \sur{} \hd(M)$, $m \mapsto m + \rad(M)$ is an essential surjection.
\end{prop}

Observe that any homomorphism of $kV_{4}$-modules $\phi \vcentcolon M \to N$ induces a map of heads $\bar{\phi} \vcentcolon \hd(M) \to \hd(N),\, m+\rad(M) \mapsto \phi(m)+\rad(N)$, since $\phi(\rad(M)) \subseteq \rad(N)$.

\begin{prop}
\label{prop:esshd}
	$\phi \text{ is an essential surjection} \iff \bar{\phi} \text{ is an isomorphism}$. Moreover, when $\phi$ (and hence $\bar{\phi}$) is known to be a surjection, the Rank-Nullity Theorem gives:
	\[
	\phi \text{ is essential} \iff \dim_{k}(\hd(M)) = \dim_{k}(\hd(N))
	\]
\end{prop}

Heller shifts for the indecomposables are given in \cite[Lem. 2.6, Prop. 3.1]{CM12}.

\begin{lem}
	The non-projective even-dimensional indecomposables are fixed under the Heller operator. The odd-dimensional case is described by the following picture:
	\[\begin{tikzcd}[cramped,sep=scriptsize]
		\cdots & {M_{5}} & {M_{3}} & k & {W_{3}} & {W_{5}} & \cdots
		\arrow["\Omega", maps to, from=1-1, to=1-2]
		\arrow["\Omega", maps to, from=1-2, to=1-3]
		\arrow["\Omega", maps to, from=1-3, to=1-4]
		\arrow["\Omega", maps to, from=1-4, to=1-5]
		\arrow["\Omega", maps to, from=1-5, to=1-6]
		\arrow["\Omega", maps to, from=1-6, to=1-7]
	\end{tikzcd}\]
	Explicitly, $\Omega(M_{3}) \cong k$,\, $\Omega(k) \cong W_{3}$,\, $\Omega(W_{3}) \cong W_{5}$ and for each $n \geq 2$, we have $\Omega(M_{2n+1}) \cong M_{2n-1}$ and $\Omega(W_{2n+1}) \cong W_{2n+3}$. For completeness, we note that $\Omega(kV_{4}) \cong 0$.
\end{lem}

\subsection{Dimension Zero}
\label{sec:dim0}
\mbox{}\\
We now identify the indecomposables with permutation dimension equal to zero, which are of course nothing but the indecomposable permutation modules. We will find these to be $k$, $E_{t}$, $E_{t+1_{k}}$, $E_{\infty}$ and $kV_{4}$, but first we review their diagrams:

\begin{nota}
	The current diagrams for $k$, $E_{t}$, $E_{t+1_{k}}$, $E_{\infty}$ and $kV_{4}$ are:
	%
	%
	%
	%
	%
	
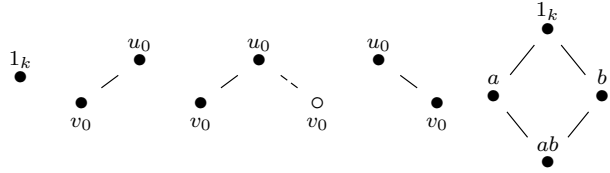
\[\begin{tikzcd}[cramped,sep=tiny]
		{\overset{1_{k}}\bullet}
	\end{tikzcd}
	\hspace{5mm}
	\begin{tikzcd}[cramped,sep=tiny]
		& {\overset{u_{0}}\bullet} \\
		{\underset{v_{0}}\bullet}
		\arrow[no head, from=1-2, to=2-1]
	\end{tikzcd}
	\hspace{5mm}
	\begin{tikzcd}[cramped,sep=tiny]
		& {\overset{u_{0}}\bullet} \\
		{\underset{v_{0}}\bullet} && {\underset{v_0}\circ}
		\arrow[no head, from=1-2, to=2-1]
		\arrow[dashed, no head, from=1-2, to=2-3]
	\end{tikzcd}
	\hspace{5mm}
	\begin{tikzcd}[cramped,sep=tiny]
		{\overset{u_{0}}\bullet} \\
		& {\underset{v_{0}}\bullet}
		\arrow[no head, from=1-1, to=2-2]
	\end{tikzcd}
	\hspace{5mm}
	\begin{tikzcd}[cramped,sep=tiny]
		& {\overset{1_{k}}{\bullet}} \\
		{\overset{a}{\bullet}} && {\overset{b}{\bullet}} \\
		& {\overset{ab}{\bullet}}
		\arrow[no head, from=1-2, to=2-1]
		\arrow[no head, from=1-2, to=2-3]
		\arrow[no head, from=2-1, to=3-2]
		\arrow[no head, from=2-3, to=3-2]
	\end{tikzcd}\]
	We introduce a new standard labelling for the bases and improve the diagram for $E_{t+1_{k}}$:
	%
	%
	%
	%
	%
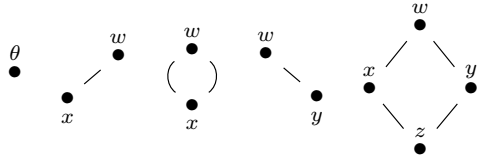
	\[\begin{tikzcd}[cramped,sep=tiny]
		{\overset{\theta}\bullet}
	\end{tikzcd}
	\hspace{5mm}
	\begin{tikzcd}[cramped,sep=tiny]
		& {\overset{w}\bullet} \\
		{\underset{x}\bullet}
		\arrow[no head, from=1-2, to=2-1]
	\end{tikzcd}
	\hspace{5mm}
	\begin{tikzcd}[cramped,sep=small]
		{\overset{w}\bullet} \\
		{\underset{x}\bullet}
		\arrow[curve={height=10pt}, no head, from=1-1, to=2-1]
		\arrow[curve={height=-10pt}, no head, from=1-1, to=2-1]
	\end{tikzcd}
	\hspace{5mm}
	\begin{tikzcd}[cramped,sep=tiny]
		{\overset{w}\bullet} \\
		& {\underset{y}\bullet}
		\arrow[no head, from=1-1, to=2-2]
	\end{tikzcd}
	\hspace{5mm}
	\begin{tikzcd}[cramped,sep=tiny]
		& {\overset{w}{\bullet}} \\
		{\overset{x}{\bullet}} && {\overset{y}{\bullet}} \\
		& {\overset{z}{\bullet}}
		\arrow[no head, from=1-2, to=2-1]
		\arrow[no head, from=1-2, to=2-3]
		\arrow[no head, from=2-1, to=3-2]
		\arrow[no head, from=2-3, to=3-2]
	\end{tikzcd}\]
\end{nota}

\begin{thm}
\label{thm:dim0}
	The indecomposable permutation modules are:
	\[
	k \,, \; E_{t} \,, \; E_{t+1_{k}} \,, \; E_{\infty} \,, \; kV_{4}
	\]
\end{thm}

\begin{proof}
	For ease of notation, we let $G \defeq V_{4}$ in this proof.
	
	An indecomposable permutation module $P$ admits a $G$-invariant $k$-basis $X \subseteq P$, which is naturally a transitive $G$-set. Indeed, a decomposition of $X$ in to $G$-orbits $\coprod_{i = 1}^{n}\left(X_{i}\right)$ gives a decomposition $P = k[X] = \bigoplus_{i = 1}^{n}\left(k\left[X_{i}\right]\right)$, hence $n=1$. Fix $x \in X$ with stabiliser subgroup $G_{x} \leq G$. Note that $x$ generates $P$ as a $kG$-module. Consider the morphism $\phi \vcentcolon P \to k\left[G/G_{x}\right],\, x \mapsto eG_{x}$. It is not hard to see that $\phi$ is a well-defined surjection, from which it follows that $\phi$ is an isomorphism by invoking both the Orbit-Stabiliser Theorem and the Rank-Nullity Theorem. The group $G \defeq V_{4}$ admits five subgroups, namely $\langle \sigma, \tau \rangle$, $\langle \tau \rangle$, $\langle \sigma\tau \rangle$, $\langle \sigma \rangle$ and $\langle e \rangle$, which give rise to five possible indecomposable permutation modules $k\left[G/\langle \sigma, \tau \rangle\right]$, $k\left[G/\langle \tau \rangle\right]$, $k\left[G/\langle \sigma\tau \rangle\right]$, $k\left[G/\langle \sigma \rangle\right]$ and $k\left[G/\langle e \rangle\right]$. It remains to observe that these correspond to $k$, $E_{t}$, $E_{t+1_{k}}$, $E_{\infty}$ and $kV_{4}$ respectively. This is a straightforward check, which may be facilitated by contemplating the following diagrams:
	%
	%
	%
	
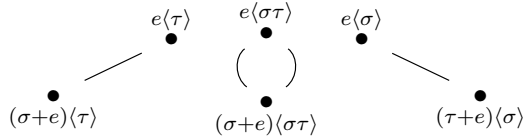
\[\begin{tikzcd}[cramped,sep=small]
		& {\overset{e \langle \tau \rangle}\bullet} \\
		{\underset{(\sigma + e)\langle \tau \rangle}\bullet}
		\arrow[no head, from=1-2, to=2-1]
	\end{tikzcd}
	\hspace{3mm}
	\begin{tikzcd}[cramped,sep=scriptsize]
		{\overset{e \langle \sigma\tau \rangle}\bullet} \\
		{\underset{(\sigma + e)\langle \sigma\tau \rangle}\bullet}
		\arrow[curve={height=12pt}, no head, from=1-1, to=2-1]
		\arrow[curve={height=-12pt}, no head, from=1-1, to=2-1]
	\end{tikzcd}
	\hspace{3mm}
	\begin{tikzcd}[cramped,sep=small]
		{\overset{e \langle \sigma \rangle}\bullet} \\
		& {\underset{(\tau + e)\langle \sigma \rangle}\bullet}
		\arrow[no head, from=1-1, to=2-2]
	\end{tikzcd}\]
\end{proof}

\section{Resolutions}
\label{sec:resol}
We now exhibit permutation resolutions for each of the indecomposable $kV_{4}$-modules, which we will eventually find to be minimal. Establishing minimality for resolutions of length two or more is not easy in general; this is the subject of Section \ref{sec:minim}.

\subsection{Dimension One}
\label{sec:dim1}
\mbox{}\\
Any length-one permutation resolution of a non-permutation module is automatically minimal. We will now identify eight indecomposables with permutation dimension one by finding such resolutions. As we will see later, these eight modules are in fact \emph{all} such indecomposables. For a $kV_{4}$-module $M$, a length-one permutation resolution is nothing but a short exact sequence of the form:
\[\begin{tikzcd}[cramped,sep=scriptsize]
	{\ker(\phi)} & P & M
	\arrow[hook, from=1-1, to=1-2]
	\arrow["\phi", two heads, from=1-2, to=1-3]
\end{tikzcd}\]
such that both $P$ and $\ker(\phi)$ are permutation modules. Note that any permutation module decomposes as a direct sum of indecomposable permutation modules identified in Theorem \ref{thm:dim0}. Each of these is principally generated, so the map $\phi$ is entirely specified by the images of the generators for each indecomposable summand of $P$.

Of course in general, we cannot map these generators arbitrarily; their images must satisfy certain relations. For example, a map out of the trivial module $\phi \vcentcolon k \to M,\, \theta \mapsto \phi(\theta)$ is well-defined if and only if  $a\phi(\theta) = b\phi(\theta) = 0_{M}$. On the other hand, the regular module $kV_{4}$ is free, so any choice for the image of the generator uniquely specifies a well-defined map of $kV_{4}$-modules. The conditions on maps out of $E_{t}$, $E_{t+1_{k}}$ and $E_{\infty}$ are as follows:

\begin{enumerate}[label=(\roman*)]
	\item The map $\phi \vcentcolon E_{t} \to M,\, w \mapsto \phi(w)$ is well-defined if and only if $b\phi(w) = 0_{M}$.
	\item The map $\phi \vcentcolon E_{t+1_{k}} \to M,\, w \mapsto \phi(w)$ is well-defined if and only if $(a + b)\phi(w) = 0_{M}$.
	\item The map $\phi \vcentcolon E_{\infty} \to M,\, w \mapsto \phi(w)$ is well-defined if and only if $a\phi(w) = 0_{M}$.
\end{enumerate}

We will verify that $\ker(\phi)$ is a permutation module by finding a $k$-basis on which the action of $kV_{4}$ is described by a diagram corresponding to some permutation module. This encodes an isomorphism between $\ker(\phi)$ and the permutation module. It is also useful to note that the diagram for a direct sum of indecomposables simply consists of the individual diagrams for each summand drawn next to each other.

\begin{thm}
\label{thm:dim1}
	The indecomposables with permutation dimension equal to one are:
	\[
	M_{3} \,, \; M_{5} \,, \; M_{7} \,, \; W_{3} \,, \; W_{5} \,, \; E_{t,2} \,, \; E_{(t+1_{k}),2} \,, \; E_{\infty,2}
	\]
	Their diagrams are as follows:
	%
	%
	%
	
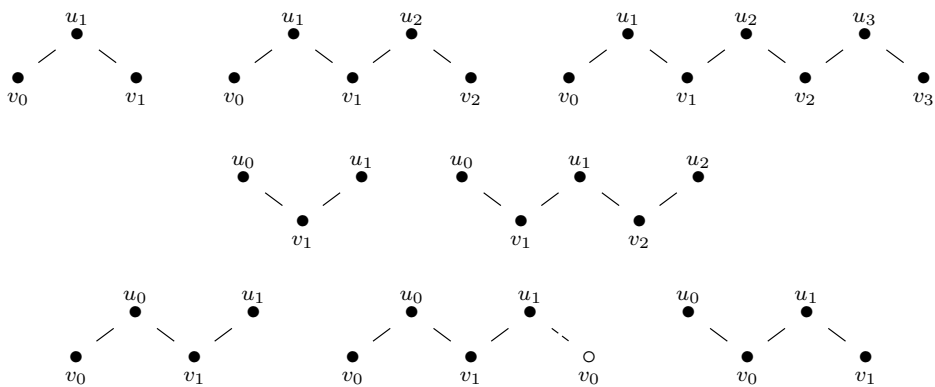
\[\begin{tikzcd}[cramped,sep=tiny]
		& {\overset{u_{1}}\bullet} \\
		{\underset{v_{0}}\bullet} && {\underset{v_{1}}\bullet}
		\arrow[no head, from=1-2, to=2-1]
		\arrow[no head, from=1-2, to=2-3]
	\end{tikzcd}
	\hspace{10mm}
	\begin{tikzcd}[cramped,sep=tiny]
		& {\overset{u_{1}}\bullet} && {\overset{u_{2}}\bullet} \\
		{\underset{v_{0}}\bullet} && {\underset{v_{1}}\bullet} && {\underset{v_{2}}\bullet}
		\arrow[no head, from=1-2, to=2-1]
		\arrow[no head, from=1-2, to=2-3]
		\arrow[no head, from=1-4, to=2-3]
		\arrow[no head, from=1-4, to=2-5]
	\end{tikzcd}
	\hspace{10mm}
	\begin{tikzcd}[cramped,sep=tiny]
		& {\overset{u_{1}}\bullet} && {\overset{u_{2}}\bullet} && {\overset{u_{3}}\bullet} \\
		{\underset{v_{0}}\bullet} && {\underset{v_{1}}\bullet} && {\underset{v_{2}}\bullet} && {\underset{v_{3}}\bullet}
		\arrow[no head, from=1-2, to=2-1]
		\arrow[no head, from=1-2, to=2-3]
		\arrow[no head, from=1-4, to=2-3]
		\arrow[no head, from=1-4, to=2-5]
		\arrow[no head, from=1-6, to=2-5]
		\arrow[no head, from=1-6, to=2-7]
	\end{tikzcd}\]
	
	%
	%
	\[\begin{tikzcd}[cramped,sep=tiny]
		{\overset{u_{0}}\bullet} && {\overset{u_{1}}\bullet} \\
		& {\underset{v_{1}}\bullet}
		\arrow[no head, from=1-1, to=2-2]
		\arrow[no head, from=1-3, to=2-2]
	\end{tikzcd}
	\hspace{10mm}
	\begin{tikzcd}[cramped,sep=tiny]
		{\overset{u_{0}}\bullet} && {\overset{u_{1}}\bullet} && {\overset{u_{2}}\bullet} \\
		& {\underset{v_{1}}\bullet} && {\underset{v_{2}}\bullet}
		\arrow[no head, from=1-1, to=2-2]
		\arrow[no head, from=1-3, to=2-2]
		\arrow[no head, from=1-3, to=2-4]
		\arrow[no head, from=2-4, to=1-5]
	\end{tikzcd}\]
	
	%
	%
	%
	\[\begin{tikzcd}[cramped,sep=tiny]
		& {\overset{u_{0}}\bullet} && {\overset{u_{1}}\bullet} \\
		{\underset{v_{0}}\bullet} && {\underset{v_{1}}\bullet}
		\arrow[no head, from=1-2, to=2-1]
		\arrow[no head, from=1-2, to=2-3]
		\arrow[no head, from=1-4, to=2-3]
	\end{tikzcd}
	\hspace{10mm}
	\begin{tikzcd}[cramped,sep=tiny]
		& {\overset{u_{0}}\bullet} && {\overset{u_{1}}\bullet} \\
		{\underset{v_{0}}\bullet} && {\underset{v_{1}}\bullet} && {\underset{v_{0}}\circ}
		\arrow[no head, from=1-2, to=2-1]
		\arrow[no head, from=1-2, to=2-3]
		\arrow[no head, from=1-4, to=2-3]
		\arrow[dashed, no head, from=1-4, to=2-5]
	\end{tikzcd}
	\hspace{10mm}
	\begin{tikzcd}[cramped,sep=tiny]
		{\overset{u_{0}}\bullet} && {\overset{u_{1}}\bullet} \\
		& {\underset{v_{0}}\bullet} && {\underset{v_{1}}\bullet}
		\arrow[no head, from=1-1, to=2-2]
		\arrow[no head, from=1-3, to=2-2]
		\arrow[no head, from=1-3, to=2-4]
	\end{tikzcd}\]
\end{thm}

\begin{proof}
	For now, we will only prove that each of these has permutation dimension equal to one. In Section \ref{sec:minim} we will see that this list is in fact complete.
	\begin{enumerate}[label=(\roman*), labelindent=0pt, labelwidth=!, wide]
		\item There is a short exact sequence:
		$\begin{tikzcd}[cramped,sep=small]
			0 & k & {kV_{4}} & {M_{3}} & 0
			\arrow[from=1-1, to=1-2]
			\arrow[from=1-2, to=1-3]
			\arrow[from=1-3, to=1-4]
			\arrow[from=1-4, to=1-5]
		\end{tikzcd}$
		\newline
		This is described explicitly as follows:
		%
		%
		%
		\[\begin{array}{ccccc}
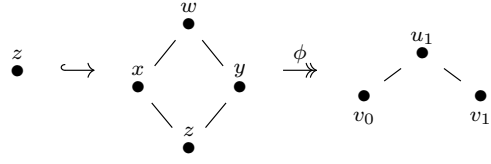
			\begin{tikzcd}[cramped,sep=tiny]
				{\overset{z}\bullet}
			\end{tikzcd}
			&
			\inj{}
			&
			\begin{tikzcd}[cramped,sep=tiny]	
				& {\overset{w}{\bullet}} \\
				{\overset{x}{\bullet}} && {\overset{y}{\bullet}} \\
				& {\overset{z}{\bullet}}
				\arrow[no head, from=1-2, to=2-1]
				\arrow[no head, from=1-2, to=2-3]
				\arrow[no head, from=2-1, to=3-2]
				\arrow[no head, from=2-3, to=3-2]
			\end{tikzcd}
			&
			\sur{\phi}
			&
			\begin{tikzcd}[cramped,sep=tiny]
				& {\overset{u_{1}}\bullet} \\
				{\underset{v_{0}}\bullet} && {\underset{v_{1}}\bullet}
				\arrow[no head, from=1-2, to=2-1]
				\arrow[no head, from=1-2, to=2-3]
			\end{tikzcd}
		\end{array}\]
		where $\phi \vcentcolon w \mapsto u_{1}$.
		\newline
		We have $\phi(x) = au_{1} = v_{0}$, $\phi(y) = bu_{1} = v_{1}$ and $\phi(z) = abu_{1} = 0_{M_{3}}$. This is a surjection since $u_{1}$ generates $M_{3}$ and $\phi$ is well-defined because $kV_{4}$ is free. $\{z\}$ is a $k$-basis for $\ker(\phi)$ since $\phi(z) = 0_{M_{3}}$ and $\dim_{k}\left(\ker(\phi)\right) = 4 - 3 = 1$. We have $az = bz = 0_{kV_{4}}$, so $\ker(\phi) \cong k$. $kV_{4}$ and $k$ are indeed permutation modules, therefore $\ppdim\left(M_{3}\right) = 1$.

		\item There is a short exact sequence:
		$\begin{tikzcd}[cramped,sep=small]
			0 & {E_{t} \oplus E_{\infty}} & {kV_{4} \oplus k \oplus kV_{4}} & {M_{5}} & 0
			\arrow[from=1-1, to=1-2]
			\arrow[from=1-2, to=1-3]
			\arrow[from=1-3, to=1-4]
			\arrow[from=1-4, to=1-5]
		\end{tikzcd}$
		\newline
		This is described explicitly as follows:
		\[\begin{array}{ccccc}
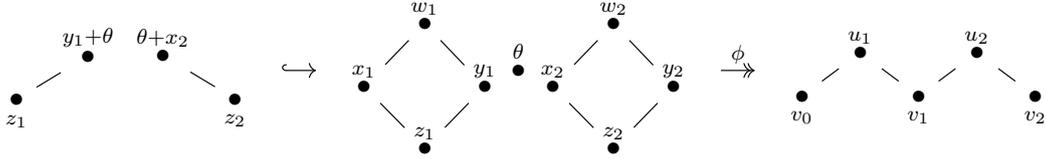
			\begin{tikzcd}[cramped,sep=tiny]
				& {\overset{y_{1}+\theta}\bullet} \\
				{\underset{z_{1}}\bullet}
				\arrow[no head, from=1-2, to=2-1]
			\end{tikzcd}
			\hspace{3mm}
			\begin{tikzcd}[cramped,sep=tiny]
				{\overset{\theta+x_{2}}\bullet} \\
				& {\underset{z_{2}}\bullet}
				\arrow[no head, from=1-1, to=2-2]
			\end{tikzcd} 
			&
			\inj{}
			&
			\begin{tikzcd}[cramped,sep=tiny]
				& {\overset{w_{1}}{\bullet}} \\
				{\overset{x_{1}}{\bullet}} && {\overset{y_{1}}{\bullet}} \\
				& {\overset{z_{1}}{\bullet}}
				\arrow[no head, from=1-2, to=2-1]
				\arrow[no head, from=1-2, to=2-3]
				\arrow[no head, from=2-1, to=3-2]
				\arrow[no head, from=2-3, to=3-2]
			\end{tikzcd}
			\hspace{2mm}
			\begin{tikzcd}[cramped,sep=tiny]
				{\overset{\theta}\bullet}
			\end{tikzcd}
			\hspace{2mm}
			\begin{tikzcd}[cramped,sep=tiny]
				& {\overset{w_{2}}{\bullet}} \\
				{\overset{x_{2}}{\bullet}} && {\overset{y_{2}}{\bullet}} \\
				& {\overset{z_{2}}{\bullet}}
				\arrow[no head, from=1-2, to=2-1]
				\arrow[no head, from=1-2, to=2-3]
				\arrow[no head, from=2-1, to=3-2]
				\arrow[no head, from=2-3, to=3-2]
			\end{tikzcd}
			&
			\sur{\phi}
			&
			\begin{tikzcd}[cramped,sep=tiny]
				& {\overset{u_{1}}\bullet} && {\overset{u_{2}}\bullet} \\
				{\underset{v_{0}}\bullet} && {\underset{v_{1}}\bullet} && {\underset{v_{2}}\bullet}
				\arrow[no head, from=1-2, to=2-1]
				\arrow[no head, from=1-2, to=2-3]
				\arrow[no head, from=1-4, to=2-3]
				\arrow[no head, from=1-4, to=2-5]
			\end{tikzcd}
		\end{array}\]
		where $\phi \vcentcolon w_{i} \mapsto u_{i},\, \theta \mapsto v_{1}$.
		\newline
		This is a surjection since $u_{1}$ and $u_{2}$ generate $M_{5}$. The map $\phi$ is well-defined because $kV_{4}$ is free and $av_{1} = bv_{1} = 0_{M_{5}}$. Note that $z_{1}, z_{2} \in \ker(\phi)$ and $\phi(y_{1}+\theta) = \phi(\theta+x_{2}) = v_{1} + v_{1} = 0_{M_{5}}$. The $k$-linearly independent set $\{z_{1},\, z_{2},\, y_{1}+\theta,\, \theta+x_{2}\}$ is a $k$-basis for $\ker(\phi)$ since $\dim_{k}(\ker(\phi)) = 9 - 5 = 4$. Each basis element is annihilated by $a$, besides $y_{1}+\theta$ which is sent to $z_{1}$. Similarly, each basis element is annihilated by $b$, besides $\theta+x_{2}$ which is sent to $z_{2}$. We therefore have $\ker(\phi) \cong E_{t} \oplus E_{\infty}$, so $\ppdim\left(M_{5}\right) = 1$.

		\item There is a short exact sequence:
		$\begin{tikzcd}[cramped,sep=small]
			0 & {E_{t} \oplus E_{t+1_{k}} \oplus E_{\infty}} & {k \oplus kV_{4}^{\oplus3}} & {M_{7}} & 0
			\arrow[from=1-1, to=1-2]
			\arrow[from=1-2, to=1-3]
			\arrow[from=1-3, to=1-4]
			\arrow[from=1-4, to=1-5]
		\end{tikzcd}$
		This is described explicitly as follows:
		\[\begin{array}{ccc}
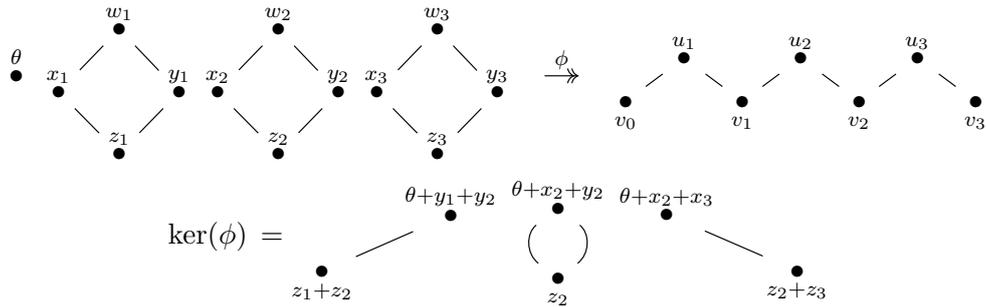
			\begin{tikzcd}[cramped,sep=tiny]
				{\overset{\theta}\bullet}
			\end{tikzcd}
			\hspace{3mm}
			\begin{tikzcd}[cramped,sep=tiny]
				& {\overset{w_{1}}{\bullet}} \\
				{\overset{x_{1}}{\bullet}} && {\overset{y_{1}}{\bullet}} \\
				& {\overset{z_{1}}{\bullet}}
				\arrow[no head, from=1-2, to=2-1]
				\arrow[no head, from=1-2, to=2-3]
				\arrow[no head, from=2-1, to=3-2]
				\arrow[no head, from=2-3, to=3-2]
			\end{tikzcd}
			\hspace{2mm}
			\begin{tikzcd}[cramped,sep=tiny]
				& {\overset{w_{2}}{\bullet}} \\
				{\overset{x_{2}}{\bullet}} && {\overset{y_{2}}{\bullet}} \\
				& {\overset{z_{2}}{\bullet}}
				\arrow[no head, from=1-2, to=2-1]
				\arrow[no head, from=1-2, to=2-3]
				\arrow[no head, from=2-1, to=3-2]
				\arrow[no head, from=2-3, to=3-2]
			\end{tikzcd}
			\hspace{2mm}
			\begin{tikzcd}[cramped,sep=tiny]
				& {\overset{w_{3}}{\bullet}} \\
				{\overset{x_{3}}{\bullet}} && {\overset{y_{3}}{\bullet}} \\
				& {\overset{z_{3}}{\bullet}}
				\arrow[no head, from=1-2, to=2-1]
				\arrow[no head, from=1-2, to=2-3]
				\arrow[no head, from=2-1, to=3-2]
				\arrow[no head, from=2-3, to=3-2]
			\end{tikzcd}
			&
			\sur{\phi}
			&
			\begin{tikzcd}[cramped,sep=tiny]
				& {\overset{u_{1}}\bullet} && {\overset{u_{2}}\bullet} && {\overset{u_{3}}\bullet} \\
				{\underset{v_{0}}\bullet} && {\underset{v_{1}}\bullet} && {\underset{v_{2}}\bullet} && {\underset{v_{3}}\bullet}
				\arrow[no head, from=1-2, to=2-1]
				\arrow[no head, from=1-2, to=2-3]
				\arrow[no head, from=1-4, to=2-3]
				\arrow[no head, from=1-4, to=2-5]
				\arrow[no head, from=1-6, to=2-5]
				\arrow[no head, from=1-6, to=2-7]
			\end{tikzcd}
		\end{array}\]
		%
		%
		%
		\[\ker(\phi) \, =
		\begin{tikzcd}[cramped,sep=small]
			& {\overset{\theta+y_{1}+y_{2}}\bullet} \\
			{\underset{z_{1}+z_{2}}\bullet}
			\arrow[no head, from=1-2, to=2-1]
		\end{tikzcd}
		\hspace{2mm}
		\begin{tikzcd}[cramped,sep=scriptsize]
			{\overset{\theta+x_{2}+y_{2}}\bullet} \\
			{\underset{z_{2}}\bullet}
			\arrow[curve={height=-12pt}, no head, from=1-1, to=2-1]
			\arrow[curve={height=12pt}, no head, from=1-1, to=2-1]
		\end{tikzcd}
		\hspace{2mm}
		\begin{tikzcd}[cramped,sep=small]
			{\overset{\theta+x_{2}+x_{3}}\bullet} \\
			& {\underset{z_{2}+z_{3}}\bullet}
			\arrow[no head, from=1-1, to=2-2]
		\end{tikzcd}
		\]
		where $\phi \vcentcolon w_{i} \mapsto u_{i},\, \theta \mapsto v_{1} + v_{2}$.
		\newline
		Since the arguments involved are all similar, the finer details are left to the reader in the remaining cases. The necessary checks can be quickly carried out by inspecting the diagrams. Indeed, it is straightforward to verify that $\phi$ is a well-defined surjection and that $\{z_{1} + z_{2},\, z_{2},\, z_{2} + z_{3},\, \theta + y_{1} + y_{2},\, \theta + x_{2} + y_{2},\, \theta + x_{2} + x_{3}\}$ is a $k$-linearly independent subset of $\ker(\phi)$. Since $\dim_{k}(\ker(\phi)) = 13 - 7 = 6$, this is a $k$-basis. Finally, we check that the $kV_{4}$-action on this basis agrees with the above diagram. This tells us that $\ker(\phi) \cong E_{t} \oplus E_{t+1_{k}} \oplus E_{\infty}$, hence $\ppdim\left(M_{7}\right) = 1$.

		\item There is a short exact sequence:
		$\begin{tikzcd}[cramped,sep=small]
			0 & k & {E_{\infty} \oplus E_{t}} & {W_{3}} & 0
			\arrow[from=1-1, to=1-2]
			\arrow[from=1-2, to=1-3]
			\arrow[from=1-3, to=1-4]
			\arrow[from=1-4, to=1-5]
		\end{tikzcd}$
		\newline
		This is described explicitly as follows:
		%
		%
		%
		%
		\[\begin{array}{ccccc}
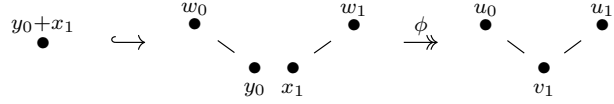
			\begin{tikzcd}[cramped,sep=tiny]
				{\overset{y_{0}+x_{1}}\bullet}
			\end{tikzcd}
			&
			\inj{}
			&
			\begin{tikzcd}[cramped,sep=tiny]
				{\overset{w_{0}}{\bullet}} \\
				& {\underset{y_{0}}{\bullet}}
				\arrow[no head, from=1-1, to=2-2]
			\end{tikzcd}
			\hspace{2mm}
			\begin{tikzcd}[cramped,sep=tiny]
				& {\overset{w_{1}}{\bullet}} \\
				{\underset{x_{1}}{\bullet}}
				\arrow[no head, from=1-2, to=2-1]
			\end{tikzcd}
			&
			\sur{\phi}
			&
			\begin{tikzcd}[cramped,sep=tiny]
				{\overset{u_{0}}\bullet} && {\overset{u_{1}}\bullet} \\
				& {\underset{v_{1}}\bullet}
				\arrow[no head, from=1-1, to=2-2]
				\arrow[no head, from=1-3, to=2-2]
			\end{tikzcd}
		\end{array}\]
		where $\phi \vcentcolon w_{i} \mapsto u_{i}$.
		\newline
		It is straightforward to check that $\phi$ is a well-defined surjection and that $\{y_{0}+x_{1}\}$ is a $k$-basis for $\ker(\phi)$. It follows that $\ker(\phi) \cong k$, hence $\ppdim\left(W_{3}\right) = 1$.

		\item There is a short exact sequence:
		$\begin{tikzcd}[cramped,sep=small]
			0 & k & {E_{\infty} \oplus E_{t+1_{k}} \oplus E_{t}} & {W_{5}} & 0
			\arrow[from=1-1, to=1-2]
			\arrow[from=1-2, to=1-3]
			\arrow[from=1-3, to=1-4]
			\arrow[from=1-4, to=1-5]
		\end{tikzcd}$
		\newline
		This is described explicitly as follows:
		%
		%
		%
		%
		%
		\[\begin{array}{ccccc}
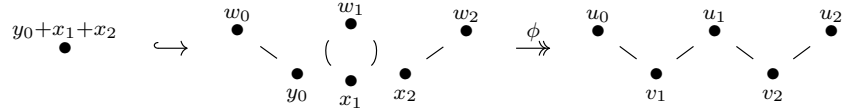
			\begin{tikzcd}[cramped,sep=tiny]
				{\overset{y_{0}+x_{1}+x_{2}}\bullet}
			\end{tikzcd}
			&
			\inj{}
			&
			\begin{tikzcd}[cramped,sep=tiny]
				{\overset{w_{0}}{\bullet}} \\
				& {\underset{y_{0}}{\bullet}}
				\arrow[no head, from=1-1, to=2-2]
			\end{tikzcd}
			\hspace{2mm}
			\begin{tikzcd}[cramped,sep=small]
				{\overset{w_{1}}\bullet} \\
				{\underset{x_{1}}\bullet}
				\arrow[curve={height=-10pt}, no head, from=1-1, to=2-1]
				\arrow[curve={height=10pt}, no head, from=1-1, to=2-1]
			\end{tikzcd}
			\hspace{2mm}
			\begin{tikzcd}[cramped,sep=tiny]
				& {\overset{w_{2}}{\bullet}} \\
				{\underset{x_{2}}{\bullet}}
				\arrow[no head, from=1-2, to=2-1]
			\end{tikzcd}
			&
			\sur{\phi}
			&
			\begin{tikzcd}[cramped,sep=tiny]
				{\overset{u_{0}}\bullet} && {\overset{u_{1}}\bullet} && {\overset{u_{2}}\bullet} \\
				& {\underset{v_{1}}\bullet} && {\underset{v_{2}}\bullet}
				\arrow[no head, from=1-1, to=2-2]
				\arrow[no head, from=1-3, to=2-2]
				\arrow[no head, from=1-3, to=2-4]
				\arrow[no head, from=1-5, to=2-4]
			\end{tikzcd}
		\end{array}\]
		where $\phi \vcentcolon w_{0} \mapsto u_{0},\, w_{1} \mapsto u_{0} + u_{1} + u_{2},\, w_{2} \mapsto u_{2}$.
		\newline
		Observe that $\phi$ is a well-defined surjection. Notably, $\{\phi(w_{0}),\, \phi(w_{1}),\, \phi(w_{2})\}$ is a generating set for $W_{5}$ and $(a + b)\phi(w_{1}) = (v_{1}+v_{2})+(v_{1}+v_{2}) = 0_{W_{5}}$. Also, $\{y_{0}+x_{1}+x_{2}\}$ is a $k$-basis for $\ker(\phi)$ since $\phi(x_{1}) = v_{1}+v_{2} = \phi(y_{0})+\phi(x_{2})$ and $\dim_{k}(\ker(\phi)) = 6 - 5 = 1$. It follows that $\ker(\phi) \cong k$, hence $\ppdim\left(W_{5}\right) = 1$.

		\item There is a short exact sequence:
		$\begin{tikzcd}[cramped,sep=small]
			0 & {E_{t}} & {kV_{4} \oplus E_{t}} & {E_{t,2}} & 0
			\arrow[from=1-1, to=1-2]
			\arrow[from=1-2, to=1-3]
			\arrow[from=1-3, to=1-4]
			\arrow[from=1-4, to=1-5]
		\end{tikzcd}$
		\newline
		This is described explicitly as follows:
		%
		%
		%
		%
		\[\begin{array}{ccccc}
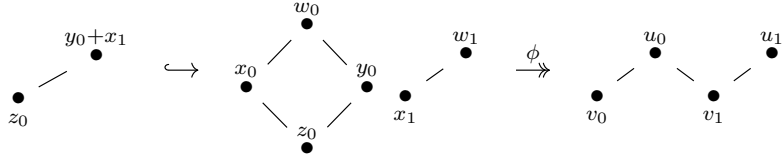
			\begin{tikzcd}[cramped,sep=tiny]
				& {\overset{y_{0}+x_{1}}{\bullet}} \\
				{\underset{z_{0}}{\bullet}}
				\arrow[no head, from=1-2, to=2-1]
			\end{tikzcd} 
			&
			\inj{}
			&
			\begin{tikzcd}[cramped,sep=tiny]
				& {\overset{w_{0}}{\bullet}} \\
				{\overset{x_{0}}{\bullet}} && {\overset{y_{0}}{\bullet}} \\
				& {\overset{z_{0}}{\bullet}}
				\arrow[no head, from=1-2, to=2-1]
				\arrow[no head, from=1-2, to=2-3]
				\arrow[no head, from=2-1, to=3-2]
				\arrow[no head, from=2-3, to=3-2]
			\end{tikzcd}
			\hspace{2mm}
			\begin{tikzcd}[cramped,sep=tiny]
				& {\overset{w_{1}}{\bullet}} \\
				{\underset{x_{1}}{\bullet}}
				\arrow[no head, from=1-2, to=2-1]
			\end{tikzcd}
			&
			\sur{\phi}
			&
			\begin{tikzcd}[cramped,sep=tiny]
				& {\overset{u_{0}}\bullet} && {\overset{u_{1}}\bullet} \\
				{\underset{v_{0}}\bullet} && {\underset{v_{1}}\bullet}
				\arrow[no head, from=1-2, to=2-1]
				\arrow[no head, from=1-2, to=2-3]
				\arrow[no head, from=1-4, to=2-3]
			\end{tikzcd}
		\end{array}\]
		where $\phi \vcentcolon w_{i} \mapsto u_{i}$.
		\newline
		Observe that $\phi$ is a well-defined surjection and that $\{y_{0}+x_{1},\, z_{0}\}$ is a $k$-linearly independent subset of $\ker(\phi)$. Since $\dim_{k}(\ker(\phi)) = 6 - 4 = 2$, this is a $k$-basis. Finally, the $kV_{4}$-action on this basis agrees with the diagram above, so $\ker(\phi) \cong E_{t}$ and $\ppdim\left(E_{t,2}\right) = 1$.

		\item There is a short exact sequence:
		$\begin{tikzcd}[cramped,sep=small]
			0 & {E_{t+1_{k}}} & {kV_{4} \oplus E_{t+1_{k}}} & {E_{(t+1_{k}),2}} & 0
			\arrow[from=1-1, to=1-2]
			\arrow[from=1-2, to=1-3]
			\arrow[from=1-3, to=1-4]
			\arrow[from=1-4, to=1-5]
		\end{tikzcd}$
		\newline
		This is described explicitly as follows:
		\[\begin{array}{ccccc}
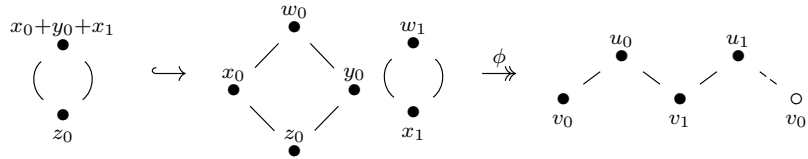
			\begin{tikzcd}[cramped,sep=scriptsize]
				{\overset{x_{0}+y_{0}+x_{1}}\bullet} \\
				{\underset{z_{0}}\bullet}
				\arrow[curve={height=12pt}, no head, from=1-1, to=2-1]
				\arrow[curve={height=-12pt}, no head, from=1-1, to=2-1]
			\end{tikzcd}
			&
			\inj{}
			&
			\begin{tikzcd}[cramped,sep=tiny]
				& {\overset{w_{0}}{\bullet}} \\
				{\overset{x_{0}}{\bullet}} && {\overset{y_{0}}{\bullet}} \\
				& {\overset{z_{0}}{\bullet}}
				\arrow[no head, from=1-2, to=2-1]
				\arrow[no head, from=1-2, to=2-3]
				\arrow[no head, from=2-1, to=3-2]
				\arrow[no head, from=2-3, to=3-2]
			\end{tikzcd}
			\hspace{2mm}
			\begin{tikzcd}[cramped,sep=scriptsize]
				{\overset{w_{1}}\bullet} \\
				{\underset{x_{1}}\bullet}
				\arrow[curve={height=12pt}, no head, from=1-1, to=2-1]
				\arrow[curve={height=-12pt}, no head, from=1-1, to=2-1]
			\end{tikzcd}
			&
			\sur{\phi}
			&
			\begin{tikzcd}[cramped,sep=tiny]
				& {\overset{u_{0}}\bullet} && {\overset{u_{1}}\bullet} \\
				{\underset{v_{0}}\bullet} && {\underset{v_{1}}\bullet} && {\underset{v_{0}}\circ}
				\arrow[no head, from=1-2, to=2-1]
				\arrow[no head, from=1-2, to=2-3]
				\arrow[no head, from=1-4, to=2-3]
				\arrow[dashed, no head, from=1-4, to=2-5]
			\end{tikzcd}
		\end{array}\]
		where $\phi \vcentcolon w_{0} \mapsto u_{0},\, w_{1} \mapsto u_{0} + u_{1}$.
		\newline
		This is indeed the correct diagram for $E_{(t+1_{k}),2}$, since $(t+1_{k})^{2} = t^{2}+1_{k}$ and so $bu_{1} = v_{0}$. Observe that $\phi$ is a well-defined surjection. Notably, $\{\phi(w_{0}),\, \phi(w_{1})\}$ is a generating set for $E_{(t+1_{k}),2}$ and $(a+b)\phi(w_{1}) = (v_{0}+v_{1})+(v_{1}+v_{0}) = 0_{E_{(t+1_{k}),2}}$. Also, $\phi(x_{1}) = v_{0}+v_{1} = \phi(x_{0})+\phi(y_{0})$, so $\{x_{0}+y_{0}+x_{1},\, z_{0}\}$ is a $k$-linearly independent subset of $\ker(\phi)$. This is a $k$-basis since $\dim_{k}(\ker(\phi)) = 6 - 4 = 2$. Finally, the $kV_{4}$-action on this basis agrees with the diagram above, so $\ker(\phi) \cong E_{t+1_{k}}$ and $\ppdim\left(E_{(t+1_{k}),2}\right) = 1$.

		\item There is a short exact sequence:
		$\begin{tikzcd}[cramped,sep=small]
			0 & {E_{\infty}} & { E_{\infty} \oplus kV_{4}} & {E_{\infty,2}} & 0
			\arrow[from=1-1, to=1-2]
			\arrow[from=1-2, to=1-3]
			\arrow[from=1-3, to=1-4]
			\arrow[from=1-4, to=1-5]
		\end{tikzcd}$
		\newline
		This is described explicitly as follows:
		%
		%
		%
		%
		\[\begin{array}{ccccc}
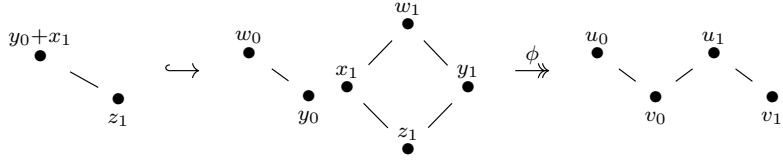
			\begin{tikzcd}[cramped,sep=tiny]
				{\overset{y_{0}+x_{1}}{\bullet}} \\
				& {\underset{z_{1}}{\bullet}}
				\arrow[no head, from=1-1, to=2-2]
			\end{tikzcd}
			&
			\inj{}
			&
			\begin{tikzcd}[cramped,sep=tiny]
				{\overset{w_{0}}{\bullet}} \\
				& {\underset{y_{0}}{\bullet}}
				\arrow[no head, from=1-1, to=2-2]
			\end{tikzcd}
			\hspace{2mm}
			\begin{tikzcd}[cramped,sep=tiny]
				& {\overset{w_{1}}{\bullet}} \\
				{\overset{x_{1}}{\bullet}} && {\overset{y_{1}}{\bullet}} \\
				& {\overset{z_{1}}{\bullet}}
				\arrow[no head, from=1-2, to=2-1]
				\arrow[no head, from=1-2, to=2-3]
				\arrow[no head, from=2-1, to=3-2]
				\arrow[no head, from=2-3, to=3-2]
			\end{tikzcd}
			&
			\sur{\phi}
			&
			\begin{tikzcd}[cramped,sep=tiny]
				{\overset{u_{0}}\bullet} && {\overset{u_{1}}\bullet} \\
				& {\underset{v_{0}}\bullet} && {\underset{v_{1}}\bullet}
				\arrow[no head, from=1-1, to=2-2]
				\arrow[no head, from=1-3, to=2-2]
				\arrow[no head, from=1-3, to=2-4]
			\end{tikzcd}
		\end{array}\]
		where $\phi \vcentcolon w_{i} \mapsto u_{i}$.
		\newline
		Observe that $\phi$ is a well-defined surjection and that $\{y_{0}+x_{1},\, z_{1}\}$ is a $k$-linearly independent subset of $\ker(\phi)$. Since $\dim_{k}(\ker(\phi)) = 6 - 4 = 2$, this is a $k$-basis. Finally, the $kV_{4}$-action on this basis agrees with the diagram above, so $\ker(\phi) \cong E_{\infty}$ and $\ppdim\left(E_{\infty,2}\right) = 1$.
		\qedhere
	\end{enumerate}
\end{proof}

\subsection{The Global Permutation Dimension}
\label{sec:glob}
\mbox{}\\
As it turns out, any $kV_{4}$-module admits a permutation resolution of length at most two. It suffices to exhibit such resolutions for all of the remaining indecomposables.

\begin{rem}
\label{rem:ppdimsum}
	It is useful to note that for a direct sum of $kV_{4}$-modules $\bigoplus_{i = 1}^{n}\left(N_{i}\right)$ we always have $\ppdim\left(\bigoplus_{i = 1}^{n}\left(N_{i}\right)\right) \leq \max \left\{\ppdim(N_{i}) \vcentcolon 1 \leq i \leq n \right\}$. Indeed, there exists a permutation resolution for each $N_{i}$ of length $l_{i} \defeq \ppdim(N_{i})$ and their degree-wise sum gives a permutation resolution for $\bigoplus_{i = 1}^{n}\left(N_{i}\right)$ of length $l \defeq \max \left\{l_{i} \vcentcolon 1 \leq i \leq n \right\}$. If each indecomposable has permutation dimension at most two, it follows that $\ppdim_{k}(V_{4}) \leq 2$.
\end{rem}

\begin{rem}
\label{rem:dim1ker}
	Suppose we find a surjection $\phi \vcentcolon P \sur{} N$ such that $P$ is a permutation module and $\ppdim\left(\ker(\phi)\right) = 1$ with permutation resolution $\ker(\phi') \inj{\iota'} P' \sur{\phi'} \ker(\phi)$. We obtain a length-two permutation resolution for $N$ by splicing short exact 
	sequences:
	\[\begin{tikzcd}[cramped,sep=normal]
		{\ker(\phi')} & {P'} && P & N \\
		&& {\ker(\phi)}
		\arrow["{\iota'}", hook, from=1-1, to=1-2]
		\arrow["{\iota \circ \phi'}", from=1-2, to=1-4]
		\arrow["{\phi'}"', two heads, from=1-2, to=2-3]
		\arrow["\phi", two heads, from=1-4, to=1-5]
		\arrow["\iota"', hook, from=2-3, to=1-4]
	\end{tikzcd}\]
	Therefore $\ppdim(N) \leq 2$. This will be our general strategy.
\end{rem}

\begin{prop}
\label{prop:ppdimM}
	For each $n \geq 4$, we have $\ppdim\left(M_{2n+1}\right) \leq 2$.
\end{prop}

\begin{proof}
	There exists a short exact sequence:
	\[\begin{tikzcd}[cramped,sep=small]
		0 & {W_{3}^{\oplus n}} & {k^{\oplus (n+1)} \oplus kV_{4}^{\oplus n}} & {M_{2n+1}} & 0
		\arrow[from=1-1, to=1-2]
		\arrow[from=1-2, to=1-3]
		\arrow[from=1-3, to=1-4]
		\arrow[from=1-4, to=1-5]
	\end{tikzcd}\]
	This is described explicitly as follows:
	\[\begin{array}{ccc}
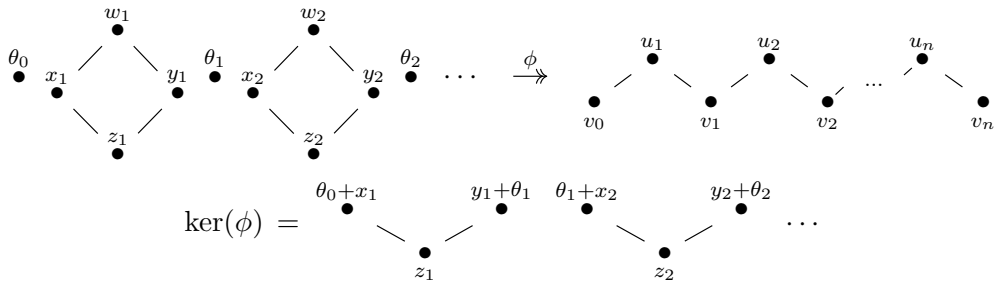
		\begin{tikzcd}[cramped,sep=tiny]
			{\overset{\theta_{0}}\bullet}
		\end{tikzcd}
		\hspace{2mm}
		\begin{tikzcd}[cramped,sep=tiny]
			& {\overset{w_{1}}{\bullet}} \\
			{\overset{x_{1}}{\bullet}} && {\overset{y_{1}}{\bullet}} \\
			& {\overset{z_{1}}{\bullet}}
			\arrow[no head, from=1-2, to=2-1]
			\arrow[no head, from=1-2, to=2-3]
			\arrow[no head, from=2-1, to=3-2]
			\arrow[no head, from=2-3, to=3-2]
		\end{tikzcd}
		\hspace{2mm}
		\begin{tikzcd}[cramped,sep=tiny]
			{\overset{\theta_{1}}\bullet}
		\end{tikzcd}
		\hspace{2mm}
		\begin{tikzcd}[cramped,sep=tiny]
			& {\overset{w_{2}}{\bullet}} \\
			{\overset{x_{2}}{\bullet}} && {\overset{y_{2}}{\bullet}} \\
			& {\overset{z_{2}}{\bullet}}
			\arrow[no head, from=1-2, to=2-1]
			\arrow[no head, from=1-2, to=2-3]
			\arrow[no head, from=2-1, to=3-2]
			\arrow[no head, from=2-3, to=3-2]
		\end{tikzcd}
		\hspace{2mm}
		\begin{tikzcd}[cramped,sep=tiny]
			{\overset{\theta_{2}}\bullet}
		\end{tikzcd}
		\;\;\cdots
		&
		\sur{\phi}
		&
		\begin{tikzcd}[cramped,sep=tiny]
			& {\overset{u_{1}}\bullet} && {\overset{u_{2}}\bullet} && {} & {\overset{u_{n}}\bullet} \\
			{\underset{v_{0}}\bullet} && {\underset{v_{1}}\bullet} && {\underset{v_{2}}\bullet} & {} && {\underset{v_{n}}\bullet}
			\arrow[no head, from=1-2, to=2-1]
			\arrow[no head, from=1-2, to=2-3]
			\arrow[no head, from=1-4, to=2-3]
			\arrow[no head, from=1-4, to=2-5]
			\arrow["\ldots"{description, pos=0.4}, draw=none, from=1-6, to=2-6]
			\arrow[shorten >=18pt, no head, from=1-7, to=2-6]
			\arrow[no head, from=1-7, to=2-8]
			\arrow[shorten >=15pt, no head, from=2-5, to=1-6]
		\end{tikzcd}
	\end{array}\]
	%
	%
	\[\ker(\phi) \,=\;
	\begin{tikzcd}[cramped,sep=tiny]
		{\overset{\theta_{0}+x_{1}}\bullet} && {\overset{y_{1}+\theta_{1}}\bullet} \\
		& {\underset{z_{1}}\bullet}
		\arrow[no head, from=1-1, to=2-2]
		\arrow[no head, from=1-3, to=2-2]
	\end{tikzcd}
	\hspace{3mm}
	\begin{tikzcd}[cramped,sep=tiny]
		{\overset{\theta_{1}+x_{2}}\bullet} && {\overset{y_{2}+\theta_{2}}\bullet} \\
		& {\underset{z_{2}}\bullet}
		\arrow[no head, from=1-1, to=2-2]
		\arrow[no head, from=1-3, to=2-2]
	\end{tikzcd}
	\;\cdots
	\]
	where $\phi \vcentcolon w_{i} \mapsto u_{i},\, \theta_{i} \mapsto v_{i}$.
	\pagebreak
	\newline
	Observe that $\phi$ is a well-defined surjection and $\bigcup_{i=1}^{n}\left(\{\theta_{i-1}+x_{i},\, y_{i}+\theta_{i},\, z_{i}\}\right) \subseteq \ker(\phi)$ is a $k$-linearly independent subset of size $3n$, notably $\phi(\theta_{i-1}+x_{i}) = v_{i-1}+v_{i-1} = 0_{M_{2n+1}}$ and $\phi(y_{i}+\theta_{i}) = v_{i}+v_{i} = 0_{M_{2n+1}}$ for $1 \leq i \leq n$. This is a $k$-basis since $\dim_{k}(\ker(\phi)) = \linebreak (n+1+4n) - (2n+1) = 3n$. The $kV_{4}$-action on this basis agrees with the above diagram, notably $a\left(\theta_{i-1}+x_{i}\right) = b\left(y_{i}+\theta_{i}\right) = 0$ and $b\left(\theta_{i-1}+x_{i}\right) = a\left(y_{i}+\theta_{i}\right) = z_{i}$ for $1 \leq i \leq n$. It follows that $\ker(\phi) \cong W_{3}^{\oplus n}$ and hence $\ppdim(\ker(\phi)) = 1$ using Remark \ref{rem:ppdimsum}. Finally, Remark \ref{rem:dim1ker} tells us that $\ppdim\left(M_{2n+1}\right) \leq 2$.
\end{proof}

\begin{prop}
\label{prop:ppdimW}
	For each $n \geq 3$, we have $\ppdim\left(W_{2n+1}\right) \leq 2$.
\end{prop}

\begin{proof}
	There exists a short exact sequence:
	\[\begin{tikzcd}[cramped,sep=small]
		0 & {W_{3}^{\oplus(n+1)}} & {k^{\oplus n} \oplus kV_{4}^{\oplus(n+1)}} & {W_{2n+1}} & 0
		\arrow[from=1-1, to=1-2]
		\arrow[from=1-2, to=1-3]
		\arrow[from=1-3, to=1-4]
		\arrow[from=1-4, to=1-5]
	\end{tikzcd}\]
	This is described explicitly as follows:
	\[\begin{array}{ccc}
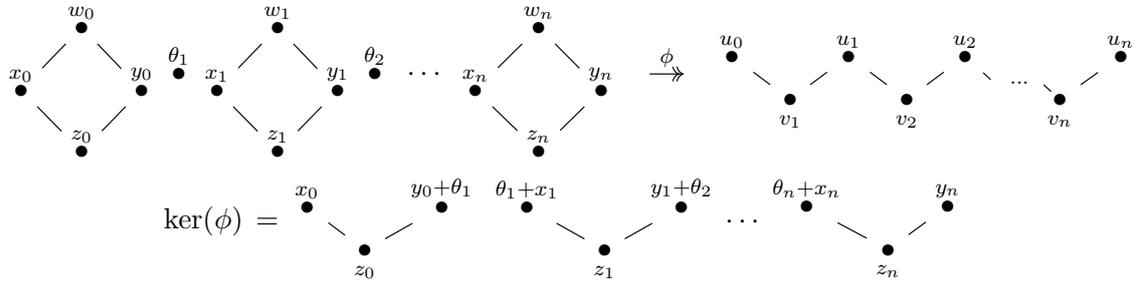
		\begin{tikzcd}[cramped,sep=tiny]
			& {\overset{w_{0}}{\bullet}} \\
			{\overset{x_{0}}{\bullet}} && {\overset{y_{0}}{\bullet}} \\
			& {\overset{z_{0}}{\bullet}}
			\arrow[no head, from=1-2, to=2-1]
			\arrow[no head, from=1-2, to=2-3]
			\arrow[no head, from=2-1, to=3-2]
			\arrow[no head, from=2-3, to=3-2]
		\end{tikzcd}
		\hspace{2mm}
		\begin{tikzcd}[cramped,sep=tiny]
			{\overset{\theta_{1}}\bullet}
		\end{tikzcd}
		\hspace{2mm}
		\begin{tikzcd}[cramped,sep=tiny]
			& {\overset{w_{1}}{\bullet}} \\
			{\overset{x_{1}}{\bullet}} && {\overset{y_{1}}{\bullet}} \\
			& {\overset{z_{1}}{\bullet}}
			\arrow[no head, from=1-2, to=2-1]
			\arrow[no head, from=1-2, to=2-3]
			\arrow[no head, from=2-1, to=3-2]
			\arrow[no head, from=2-3, to=3-2]
		\end{tikzcd}
		\hspace{2mm}
		\begin{tikzcd}[cramped,sep=tiny]
			{\overset{\theta_{2}}\bullet}
		\end{tikzcd}
		\;\;\cdots\;\;
		\begin{tikzcd}[cramped,sep=tiny]
			& {\overset{w_{n}}{\bullet}} \\
			{\overset{x_{n}}{\bullet}} && {\overset{y_{n}}{\bullet}} \\
			& {\overset{z_{n}}{\bullet}}
			\arrow[no head, from=1-2, to=2-1]
			\arrow[no head, from=1-2, to=2-3]
			\arrow[no head, from=2-1, to=3-2]
			\arrow[no head, from=2-3, to=3-2]
		\end{tikzcd}
		&
		\sur{\phi}
		&
		\begin{tikzcd}[cramped,sep=tiny]
			{\overset{u_{0}}\bullet} && {\overset{u_{1}}\bullet} && {\overset{u_{2}}\bullet} & {} && {\overset{u_{n}}\bullet} \\
			& {\underset{v_{1}}\bullet} && {\underset{v_{2}}\bullet} && {} & {\underset{v_{n}}\bullet}
			\arrow[no head, from=1-1, to=2-2]
			\arrow[no head, from=1-3, to=2-2]
			\arrow[no head, from=1-3, to=2-4]
			\arrow[shorten >=10pt, no head, from=1-5, to=2-6]
			\arrow["\ldots"{description, pos=0.4}, shift left, draw=none, from=1-6, to=2-6]
			\arrow[no head, from=1-8, to=2-7]
			\arrow[no head, from=2-4, to=1-5]
			\arrow[shorten >=15pt, no head, from=2-7, to=1-6]
		\end{tikzcd}
	\end{array}\]
	%
	%
	%
	\[\ker(\phi) \,=\;
	\begin{tikzcd}[cramped,sep=tiny]
		{\overset{x_{0}}\bullet} && {\overset{y_{0}+\theta_{1}}\bullet} \\
		& {\underset{z_{0}}\bullet}
		\arrow[no head, from=1-1, to=2-2]
		\arrow[no head, from=1-3, to=2-2]
	\end{tikzcd}
	\hspace{3mm}
	\begin{tikzcd}[cramped,sep=tiny]
		{\overset{\theta_{1}+x_{1}}\bullet} && {\overset{y_{1}+\theta_{2}}\bullet} \\
		& {\underset{z_{1}}\bullet}
		\arrow[no head, from=1-1, to=2-2]
		\arrow[no head, from=1-3, to=2-2]
	\end{tikzcd}
	\;\cdots\;
	\begin{tikzcd}[cramped,sep=tiny]
		{\overset{\theta_{n}+x_{n}}\bullet} && {\overset{y_{n}}\bullet} \\
		& {\underset{z_{n}}\bullet}
		\arrow[no head, from=1-1, to=2-2]
		\arrow[no head, from=1-3, to=2-2]
	\end{tikzcd}
	\]
	where $\phi \vcentcolon w_{i} \mapsto u_{i},\, \theta_{i} \mapsto v_{i}$.
	\newline
	Observe that $\phi$ is a well-defined surjection. The following is a $k$-linearly independent subset of $\ker(\phi)$: $\{x_{0},\, y_{0}+\theta_{1},\, z_{0}\} \cup \bigcup_{i=1}^{n-1}\left(\{\theta_{i}+x_{i},\, y_{i}+\theta_{i+1},\, z_{i}\}\right) \cup \{\theta_{n}+x_{n},\, y_{n},\, z_{n}\}$. Notably, $\phi(\theta_{i}+x_{i}) = \phi(y_{i-1}+\theta_{i}) = v_{i}+v_{i} = 0_{W_{2n+1}}$ for $1 \leq i \leq n$. This is a $k$-basis since $\dim_{k}(\ker(\phi)) = \left(n+4(n+1)\right) - (2n+1) = 3n+3$; the size of this subset. The $kV_{4}$-action on this basis agrees with the above diagram, notably $a\left(\theta_{i}+x_{i}\right) = b\left(y_{i}+\theta_{i+1}\right) = 0$ and $b\left(\theta_{i}+x_{i}\right) = a\left(y_{i}+\theta_{i+1}\right) = z_{i}$ for $1 \leq i \leq n-1$. It follows that $\ker(\phi) \cong W_{3}^{\oplus n+1}$ so $\ppdim(\ker(\phi)) = 1$ and therefore $\ppdim\left(W_{2n+1}\right) \leq 2$ by Remark \ref{rem:dim1ker}.
\end{proof}

\begin{prop}
\label{prop:ppdimE}
	Let $n \geq 1$ and either $f \in k[t]$ a monic irreducible polynomial or $f = \infty$. Then $\ppdim\left(E_{f,n}\right) \leq 2$.
\end{prop}

\begin{proof}
	Consider the case where $f \in k[t]$ is a monic irreducible polynomial. We have $f^{n} = t^{m} + \sum_{i = 0}^{m - 1}(\alpha_{i} t^{i})$, where $m \defeq n\deg(f)$ and $\alpha_{i} \in k$. There is a short exact sequence:
	\[\begin{tikzcd}[cramped,sep=small]
		0 & {W_{3}^{\oplus m}} & {k^{\oplus m} \oplus kV_{4}^{\oplus m}} & {E_{f,n}} & 0
		\arrow[from=1-1, to=1-2]
		\arrow[from=1-2, to=1-3]
		\arrow[from=1-3, to=1-4]
		\arrow[from=1-4, to=1-5]
	\end{tikzcd}\]
	This is described explicitly as follows:
	\[\begin{array}{ccc}
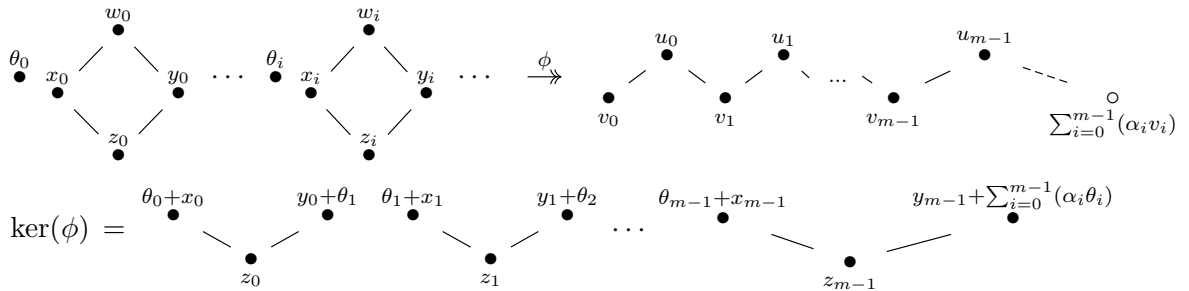
		\begin{tikzcd}[cramped,sep=tiny]
			{\overset{\theta_{0}}\bullet}
		\end{tikzcd}
		\hspace{2mm}
		\begin{tikzcd}[cramped,sep=tiny]
			& {\overset{w_{0}}{\bullet}} \\
			{\overset{x_{0}}{\bullet}} && {\overset{y_{0}}{\bullet}} \\
			& {\overset{z_{0}}{\bullet}}
			\arrow[no head, from=1-2, to=2-1]
			\arrow[no head, from=1-2, to=2-3]
			\arrow[no head, from=2-1, to=3-2]
			\arrow[no head, from=2-3, to=3-2]
		\end{tikzcd}
		\;\;\cdots\;\;
		\begin{tikzcd}[cramped,sep=tiny]
			{\overset{\theta_{i}}\bullet}
		\end{tikzcd}
		\hspace{2mm}
		\begin{tikzcd}[cramped,sep=tiny]
			& {\overset{w_{i}}{\bullet}} \\
			{\overset{x_{i}}{\bullet}} && {\overset{y_{i}}{\bullet}} \\
			& {\overset{z_{i}}{\bullet}}
			\arrow[no head, from=1-2, to=2-1]
			\arrow[no head, from=1-2, to=2-3]
			\arrow[no head, from=2-1, to=3-2]
			\arrow[no head, from=2-3, to=3-2]
		\end{tikzcd}
		\;\;\cdots
		&
		\sur{\phi}
		&
		\begin{tikzcd}[cramped,sep=tiny]
			& {\overset{u_{0}}\bullet} && {\overset{u_{1}}\bullet} & {} && {\overset{u_{m - 1}}\bullet} \\
			{\underset{v_{0}}\bullet} && {\underset{v_{1}}\bullet} && {} & {\underset{v_{m - 1}}\bullet} && {\underset{\sum_{i = 0}^{m - 1}(\alpha_{i} v_{i})}{\circ}}
			\arrow[no head, from=1-2, to=2-1]
			\arrow[no head, from=1-2, to=2-3]
			\arrow[no head, from=1-4, to=2-3]
			\arrow[shorten >=10pt, no head, from=1-4, to=2-5]
			\arrow["\ldots"{description, pos=0.4}, shift left, draw=none, from=1-5, to=2-5]
			\arrow[no head, from=1-7, to=2-6]
			\arrow[dashed, no head, from=1-7, to=2-8]
			\arrow[shorten >=10pt, no head, from=2-6, to=1-5]
		\end{tikzcd}
	\end{array}\]
	%
	%
	%
	\[\ker(\phi) \,=\;
	\begin{tikzcd}[cramped,sep=tiny]
		{\overset{\theta_{0}+x_{0}}\bullet} && {\overset{y_{0}+\theta_{1}}\bullet} \\
		& {\underset{z_{0}}\bullet}
		\arrow[no head, from=1-1, to=2-2]
		\arrow[no head, from=1-3, to=2-2]
	\end{tikzcd}
	\hspace{3mm}
	\begin{tikzcd}[cramped,sep=tiny]
		{\overset{\theta_{1}+x_{1}}\bullet} && {\overset{y_{1}+\theta_{2}}\bullet} \\
		& {\underset{z_{1}}\bullet}
		\arrow[no head, from=1-1, to=2-2]
		\arrow[no head, from=1-3, to=2-2]
	\end{tikzcd}
	\;\cdots\;
	\begin{tikzcd}[cramped,sep=tiny]
		{\overset{\theta_{m-1}+x_{m-1}}\bullet} && {\overset{y_{m-1}+\sum_{i = 0}^{m - 1}(\alpha_{i} \theta_{i})}\bullet} \\
		& {\underset{z_{m-1}}\bullet}
		\arrow[no head, from=1-1, to=2-2]
		\arrow[no head, from=1-3, to=2-2]
	\end{tikzcd}
	\]
	where $\phi \vcentcolon w_{i} \mapsto u_{i},\, \theta_{i} \mapsto v_{i}$.
	\pagebreak
	\newline
	This is indeed a well-defined surjection. The following is a $k$-linearly independent subset of $\ker(\phi)$: $\bigcup_{i=0}^{m-2}\left(\{\theta_{i}+x_{i},\, y_{i}+\theta_{i+1},\, z_{i}\}\right) \cup \{\theta_{m-1}+x_{m-1},\, y_{m-1}+\sum_{i = 0}^{m - 1}(\alpha_{i} \theta_{i}),\, z_{m-1}\}$. Notably, $\phi(\theta_{i}+x_{i}) = v_{i}+v_{i} = 0_{E_{f,n}}$ for $0 \leq i \leq m-1$, $\phi(y_{i}+\theta_{i+1}) = v_{i+1}+v_{i+1} = 0_{E_{f,n}}$ for $0 \leq i \leq m-2$ and lastly $\phi(y_{m-1}+\sum_{i = 0}^{m - 1}(\alpha_{i} \theta_{i})) = \sum_{i = 0}^{m - 1}(\alpha_{i} v_{i})+\sum_{i = 0}^{m - 1}(\alpha_{i} v_{i}) = 0_{E_{f,n}}$. This is a $k$-basis since $\dim(\ker(\phi)) = (m + 4m) - (2m) = 3m$; the size of this subset. The $kV_{4}$-action on this basis agrees with the above diagram, notably $a\left(\theta_{i}+x_{i}\right) = 0$ and $b\left(\theta_{i}+x_{i}\right) = z_{i}$ for $0 \leq i \leq m-1$, while $a\left(y_{i}+\theta_{i+1}\right) = z_{i}$ and $b\left(y_{i}+\theta_{i+1}\right) = 0$ for $0 \leq i \leq m-2$. Moreover, $a\big(y_{m-1}+\sum_{i = 0}^{m - 1}(\alpha_{i} \theta_{i})\big) = z_{m-1}$ and $b\big(y_{m-1}+\sum_{i = 0}^{m - 1}(\alpha_{i} \theta_{i})\big) = 0$. It follows that $\ker(\phi) \cong W_{3}^{\oplus m}$ so $\ppdim(\ker(\phi)) = 1$ and therefore $\ppdim\left(E_{f,n}\right) \leq 2$ by Remark \ref{rem:dim1ker}. The case $f = \infty$ is analagous, so we omit the argument. For completeness, the short exact sequence and corresponding diagrams are given below:
	\[\begin{tikzcd}[cramped,sep=small]
		0 & {W_{3}^{\oplus n}} & {kV_{4}^{\oplus n} \oplus k^{\oplus n}} & {E_{\infty,n}} & 0
		\arrow[from=1-1, to=1-2]
		\arrow[from=1-2, to=1-3]
		\arrow[from=1-3, to=1-4]
		\arrow[from=1-4, to=1-5]
	\end{tikzcd}\]
	\[\begin{array}{ccc}
		\begin{tikzcd}[cramped,sep=tiny]
			& {\overset{w_{0}}{\bullet}} \\
			{\overset{x_{0}}{\bullet}} && {\overset{y_{0}}{\bullet}} \\
			& {\overset{z_{0}}{\bullet}}
			\arrow[no head, from=1-2, to=2-1]
			\arrow[no head, from=1-2, to=2-3]
			\arrow[no head, from=2-1, to=3-2]
			\arrow[no head, from=2-3, to=3-2]
		\end{tikzcd}
		\hspace{2mm}
		\begin{tikzcd}[cramped,sep=tiny]
			{\overset{\theta_{0}}\bullet}
		\end{tikzcd}
		\;\;\cdots\;\;
		\begin{tikzcd}[cramped,sep=tiny]
			& {\overset{w_{i}}{\bullet}} \\
			{\overset{x_{i}}{\bullet}} && {\overset{y_{i}}{\bullet}} \\
			& {\overset{z_{i}}{\bullet}}
			\arrow[no head, from=1-2, to=2-1]
			\arrow[no head, from=1-2, to=2-3]
			\arrow[no head, from=2-1, to=3-2]
			\arrow[no head, from=2-3, to=3-2]
		\end{tikzcd}
		\hspace{2mm}
		\begin{tikzcd}[cramped,sep=tiny]
			{\overset{\theta_{i}}\bullet}
		\end{tikzcd}
		\;\;\cdots
		&
		\sur{\phi}
		&
		\begin{tikzcd}[cramped,sep=tiny]
			{\overset{u_{0}}\bullet} && {\overset{u_{1}}\bullet} && {} & {\overset{u_{n - 1}}\bullet} \\
			& {\underset{v_{0}}\bullet} && {\underset{v_{1}}\bullet} & {} && {\underset{v_{n - 1}}\bullet}
			\arrow[no head, from=1-1, to=2-2]
			\arrow[no head, from=1-3, to=2-2]
			\arrow[no head, from=1-3, to=2-4]
			\arrow["\ldots"{description, shift right,pos=0.4}, draw=none, from=1-5, to=2-5]
			\arrow[shorten >=12pt, no head, from=1-6, to=2-5]
			\arrow[no head, from=1-6, to=2-7]
			\arrow[shorten >=15pt, no head, from=2-4, to=1-5]
		\end{tikzcd}
	\end{array}\]
	%
	%
	%
	\[\ker(\phi) \,=\;
	\begin{tikzcd}[cramped,sep=tiny]
		{\overset{x_{0}}\bullet} && {\overset{y_{0}+\theta_{0}}\bullet} \\
		& {\underset{z_{0}}\bullet}
		\arrow[no head, from=1-1, to=2-2]
		\arrow[no head, from=1-3, to=2-2]
	\end{tikzcd}
	\hspace{3mm}
	\begin{tikzcd}[cramped,sep=tiny]
		{\overset{\theta_{0}+x_{1}}\bullet} && {\overset{y_{1}+\theta_{1}}\bullet} \\
		& {\underset{z_{1}}\bullet}
		\arrow[no head, from=1-1, to=2-2]
		\arrow[no head, from=1-3, to=2-2]
	\end{tikzcd}
	\;\cdots\;
	\begin{tikzcd}[cramped,sep=tiny]
		{\overset{\theta_{n-2}+x_{n-1}}\bullet} && {\overset{y_{n-1}+\theta_{n-1}}\bullet} \\
		& {\underset{z_{n-1}}\bullet}
		\arrow[no head, from=1-1, to=2-2]
		\arrow[no head, from=1-3, to=2-2]
	\end{tikzcd}
	\]
	where $\phi \vcentcolon w_{i} \mapsto u_{i},\, \theta_{i} \mapsto v_{i}$. Remark \ref{rem:dim1ker} then tells us that $\ppdim\left(E_{\infty,n}\right) \leq 2$.
\end{proof}

We have now shown that each indecomposable has permutation dimension at most two. Since any $kV_{4}$-module decomposes as a direct sum of indecomposables, it follows that $\ppdim_{k}(V_{4}) \leq 2$ by Remark \ref{rem:ppdimsum}. In Section \ref{sec:minim}, we establish the existence of $kV_{4}$-modules with permutation dimension equal to two, which confirms the following result:

\begin{thm}
\label{thm:glppdim}
	For a field $k$ of characteristic $p \defeq 2$, we have $\ppdim_{k}\left(V_{4}\right) = 2$.
\end{thm}

\section{Minimality}
\label{sec:minim}
We show that the indecomposables not listed in Theorems \ref{thm:dim0} and \ref{thm:dim1} have permutation dimension equal to two. Some technical results are required before we present the proof.

\subsection{Projective-Free Permutation Modules}
\label{sec:proj}
\mbox{}\\
A $kV_{4}$-module $N$ is called \emph{projective-free} if it has no non-zero projective summands, which exactly means that $kV_{4}$ does not occur as an indecomposable summand of $N$. It follows that a permutation module is projective-free if and only if it decomposes as a direct sum of copies of $k$, $E_{t}$, $E_{t+1_{k}}$ and $E_{\infty}$. Our goal in this subsection is to understand the possible submodules of projective-free permutation modules.

\begin{rem}
\label{rem:pfhom}
	Let $P$ be a projective-free permutation module and $\phi \vcentcolon P \to N$ a morphism of $kV_{4}$-modules. The image of $\phi$ is equal to the sum of the images of the indecomposable summands of $P$ under $\phi$. As discussed, these summands are (up to isomorphism) among $k$, $E_{t}$, $E_{t+1_{k}}$ and $E_{\infty}$. Since $a$ annihilates $E_{\infty}$, its image under $\phi$ must be contained in $\ker(a) \subseteq N$. Similarly, the images of $E_{t+1_{k}}$ and $E_{t}$ must be contained in $\ker(a+b)$ and $\ker(b)$ respectively. The image of $k$ will of course be contained in $\ker(a) \cap \ker(b)$. It follows that $\im(\phi) \subseteq \ker(a) + \ker(a+b) + \ker(b) \subseteq N$. It will be useful to compute $\ker(a) + \ker(a+b) + \ker(b)$ for each of the indecomposables. For the trivial module $k$ we have $\ker(a) = \ker(a+b) = \ker(b) = k$. For $kV_{4} \cong k[a,b]\,\big/\langle a^{2}, b^{2} \rangle$ one can easily check that $\ker(a) = \langle a \rangle$, $\ker(a+b) = \langle a+b \rangle$ and $\ker(b) = \langle b \rangle$, so their sum is $\langle a,b \rangle = \J(kV_{4})$.
\end{rem}

\begin{prop}
\label{prop:mker}
	Let $n \geq 1$. For $M_{2n+1}$ we have:
	\[
	\ker(a) + \ker(a+b) + \ker(b) = \vspan_{k}\{v_{i} \vcentcolon 0 \leq i \leq n\}
	\]
\end{prop}

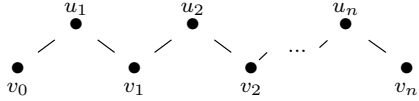
\begin{proof}
	\[\begin{tikzcd}[cramped,sep=tiny]
		& {\overset{u_{1}}\bullet} && {\overset{u_{2}}\bullet} && {} & {\overset{u_{n}}\bullet} \\
		{\underset{v_{0}}\bullet} && {\underset{v_{1}}\bullet} && {\underset{v_{2}}\bullet} & {} && {\underset{v_{n}}\bullet}
		\arrow[no head, from=1-2, to=2-1]
		\arrow[no head, from=1-2, to=2-3]
		\arrow[no head, from=1-4, to=2-3]
		\arrow[no head, from=1-4, to=2-5]
		\arrow["\ldots"{description, pos=0.4}, draw=none, from=1-6, to=2-6]
		\arrow[shorten >=10pt, no head, from=1-7, to=2-6]
		\arrow[no head, from=1-7, to=2-8]
		\arrow[shorten >=15pt, no head, from=2-5, to=1-6]
	\end{tikzcd}\]
	In fact, $\ker(a) = \ker(a+b) = \ker(b) = \vspan_{k}\{v_{i} \vcentcolon 0 \leq i \leq n\}$. Since $\{v_{i} \vcentcolon 0 \leq i \leq n\}$ is a $k$-linearly independent subset of $\ker(a)$, it suffices to show that $\dim_{k}(aM_{2n+1}) \geq n$ by the Rank-Nullity Theorem. This is clear, since $\{v_{i} \vcentcolon 0 \leq i \leq n-1\} \subseteq aM_{2n+1}$ is a $k$-linearly independent subset. A similar argument works for $\ker(a+b)$ and $\ker(b)$, noting that $\{v_{i-1} + v_{i} \vcentcolon 1 \leq i \leq n\} \subseteq (a+b)M_{2n+1}$ and $\{v_{i} \vcentcolon 1 \leq i \leq n\} \subseteq bM_{2n+1}$ are $k$-linearly independent subsets.
\end{proof}

\begin{prop}
\label{prop:wker}
	Let $n \geq 1$. For $W_{2n+1}$ we have:
	\[
	\ker(a) + \ker(a+b) + \ker(b) = \vspan_{k}\bigg(\{v_{i} \vcentcolon 1 \leq i \leq n\} \cup \big\{u_{0},\, \sum_{i=0}^{n}(u_{i}),\, u_{n}\big\}\bigg)
	\]
\end{prop}

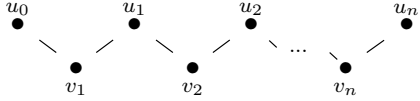
\begin{proof}
	\[\begin{tikzcd}[cramped,sep=tiny]
		{\overset{u_{0}}\bullet} && {\overset{u_{1}}\bullet} && {\overset{u_{2}}\bullet} & {} && {\overset{u_{n}}\bullet} \\
		& {\underset{v_{1}}\bullet} && {\underset{v_{2}}\bullet} && {} & {\underset{v_{n}}\bullet}
		\arrow[no head, from=1-1, to=2-2]
		\arrow[no head, from=1-3, to=2-2]
		\arrow[no head, from=1-3, to=2-4]
		\arrow[shorten >=10pt, no head, from=1-5, to=2-6]
		\arrow["\ldots"{description, pos=0.4}, draw=none, from=1-6, to=2-6]
		\arrow[no head, from=1-8, to=2-7]
		\arrow[no head, from=2-4, to=1-5]
		\arrow[shorten >=15pt, no head, from=2-7, to=1-6]
	\end{tikzcd}\]
	We have $\{v_{i} \vcentcolon 1 \leq i \leq n\} \subseteq \ker(a) \cap \ker(a+b) \cap \ker(b)$, $u_{0} \in \ker(a)$, $\sum_{i=0}^{n}(u_{i}) \in \ker(a+b)$ and $u_{n} \in \ker(b)$. Notably, $(a+b)\sum_{i=0}^{n}(u_{i}) = \sum_{i=1}^{n}(v_{i}) + \sum_{i=1}^{n}(v_{i}) = 0_{W_{2n+1}}$. We will show that the $k$-dimensions of $aW_{2n+1}$, $(a+b)W_{2n+1}$ and $bW_{2n+1}$ are at least $n$. It then follows by the Rank-Nullity Theorem that $\ker(a)$, $\ker(a+b)$ and $\ker(b)$ are respectively equal to $\vspan_{k}\big(\{v_{i} \vcentcolon 1 \leq i \leq n\} \cup \{u_{0}\}\big)$, $\vspan_{k}\big(\{v_{i} \vcentcolon 1 \leq i \leq n\} \cup \big\{\sum_{i=0}^{n}(u_{i})\big\}\big)$ and $\vspan_{k}\big(\{v_{i} \vcentcolon 1 \leq i \leq n\} \cup \{u_{n}\}\big)$. Firstly, $\{v_{i} \vcentcolon 1 \leq i \leq n\}$ is a $k$-linearly independent subset of both $aW_{2n+1}$ and $bW_{2n+1}$. For $(a+b)W_{2n+1}$, note that $(a+b)\{u_{i} \vcentcolon 1 \leq i \leq n\} \linebreak = \{v_{i} + v_{i+1} \vcentcolon 1 \leq i \leq n-1\} \cup \{v_{n}\}$ is a $k$-linearly independent subset, so we are done.
\end{proof}

\begin{prop}
\label{prop:etker}
	Let $n \geq 1$. For $E_{t,n}$ we have:
	\[
	\ker(a) + \ker(a+b) + \ker(b) = \vspan_{k}\big(\{v_{i} \vcentcolon 0 \leq i \leq n-1\} \cup \{u_{n-1}\}\big)
	\]
\end{prop}

\begin{proof}
	\[\begin{tikzcd}[cramped,sep=tiny]
		& {\overset{u_{0}}\bullet} && {\overset{u_{1}}\bullet} & {} && {\overset{u_{n - 1}}\bullet} \\
		{\underset{v_{0}}\bullet} && {\underset{v_{1}}\bullet} && {} & {\underset{v_{n - 1}}\bullet}
		\arrow[no head, from=1-2, to=2-1]
		\arrow[no head, from=1-2, to=2-3]
		\arrow[no head, from=1-4, to=2-3]
		\arrow[shorten >=10pt, no head, from=1-4, to=2-5]
		\arrow["\ldots"{description, pos=0.4}, draw=none, from=1-5, to=2-5]
		\arrow[no head, from=1-7, to=2-6]
		\arrow[shorten >=10pt, no head, from=2-6, to=1-5]
	\end{tikzcd}\]
	Note that $\{v_{i} \vcentcolon 0 \leq i \leq n-1\} \subseteq \ker(a) \cap \ker(a+b) \cap \ker(b)$ and $u_{n-1} \in \ker(b)$. We will show that $\dim_{k}(aE_{t,n}),\, \dim_{k}((a+b)E_{t,n}) \geq n$ and $\dim_{k}(bE_{t,n}) \geq n-1$, therefore $\ker(a) = \ker(a+b) = \vspan_{k}\{v_{i} \vcentcolon 0 \leq i \leq n-1\}$ and $\ker(b) = \vspan_{k}(\{v_{i} \vcentcolon 0 \leq i \leq n-1\} \cup \{u_{n-1}\})$. We have the following $k$-linearly independent subsets: $\{v_{i} \vcentcolon 0 \leq i \leq n-1\} \subseteq aE_{t,n}$, $\{v_{i} + v_{i+1} \vcentcolon 0 \leq i \leq n-2\} \cup \{v_{n-1}\} \subseteq (a+b)E_{t,n}$ and $\{v_{i} \vcentcolon 1 \leq i \leq n-1\} \subseteq bE_{t,n}$.
\end{proof}

\begin{prop}
\label{prop:eiker}
	Let $n \geq 1$. For $E_{\infty,n}$ we have:
	\[
	\ker(a) + \ker(a+b) + \ker(b) = \vspan_{k}\big(\{v_{i} \vcentcolon 0 \leq i \leq n-1\} \cup \{u_{0}\}\big)
	\]
\end{prop}

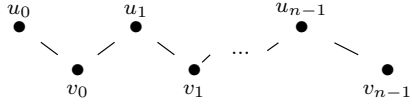
\begin{proof}
	\[\begin{tikzcd}[cramped,sep=tiny]
		{\overset{u_{0}}\bullet} && {\overset{u_{1}}\bullet} && {} & {\overset{u_{n - 1}}\bullet} \\
		& {\underset{v_{0}}\bullet} && {\underset{v_{1}}\bullet} & {} && {\underset{v_{n - 1}}\bullet}
		\arrow[no head, from=1-1, to=2-2]
		\arrow[no head, from=1-3, to=2-2]
		\arrow[no head, from=1-3, to=2-4]
		\arrow["\ldots"{description, pos=0.4}, draw=none, from=1-5, to=2-5]
		\arrow[shorten >=11pt, no head, from=1-6, to=2-5]
		\arrow[no head, from=1-6, to=2-7]
		\arrow[shorten >=15pt, no head, from=2-4, to=1-5]
	\end{tikzcd}\]
	The argument is analogous to Proposition \ref{prop:etker}.
\end{proof}

\begin{prop}
\label{prop:et1ker}
	Let $n \geq 1$. Then $(t+1_{k})^{n} = t^{n} + \sum_{i=0}^{n-1}(\alpha_{i}t^{i})$ for some $\alpha_{i} \in k$. \newline Set $\gamma_{i} \defeq \sum_{j=0}^{i}(\alpha_{j})$ and note that $\gamma_{0} = \alpha_{0} = 1_{k}$. For $E_{(t+1_{k}),n}$ we have:
	\[
	\ker(a) + \ker(a+b) + \ker(b) = \vspan_{k}\bigg(\{v_{i} \vcentcolon 0 \leq i \leq n-1\} \cup \big\{\sum_{i=0}^{n-1}(\gamma_{i}u_{i})\big\}\bigg)
	\]
\end{prop}

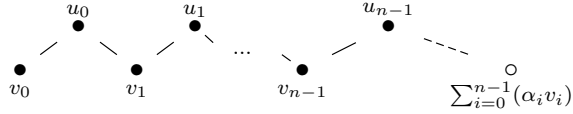
\begin{proof}
	\[\begin{tikzcd}[cramped,sep=tiny]
		& {\overset{u_{0}}\bullet} && {\overset{u_{1}}\bullet} & {} && {\overset{u_{n - 1}}\bullet} \\
		{\underset{v_{0}}\bullet} && {\underset{v_{1}}\bullet} && {} & {\underset{v_{n - 1}}\bullet} && {\underset{\sum_{i = 0}^{n - 1}(\alpha_{i} v_{i})}{\circ}}
		\arrow[no head, from=1-2, to=2-1]
		\arrow[no head, from=1-2, to=2-3]
		\arrow[no head, from=1-4, to=2-3]
		\arrow[shorten >=18pt, no head, from=1-4, to=2-5]
		\arrow["\ldots"{description, pos=0.4}, draw=none, from=1-5, to=2-5]
		\arrow[no head, from=1-7, to=2-6]
		\arrow[dashed, no head, from=1-7, to=2-8]
		\arrow[shorten >=18pt, no head, from=2-6, to=1-5]
	\end{tikzcd}\]
	Note that $\{v_{i} \vcentcolon 0 \leq i \leq n-1\} \subseteq \ker(a) \cap \ker(a+b) \cap \ker(b)$. Also, we claim that $\sum_{i=0}^{n-1}(\gamma_{i}u_{i}) \in \ker(a+b)$. We will show that $\dim_{k}(aE_{(t+1_{k}),n}),\, \dim_{k}(bE_{(t+1_{k}),n}) \geq n$ and $\dim_{k}((a+b)E_{(t+1_{k}),n}) \geq n-1$, therefore $\ker(a) = \ker(b) = \vspan_{k}\{v_{i} \vcentcolon 0 \leq i \leq n-1\}$ and $\ker(a+b) = \vspan_{k}(\{v_{i} \vcentcolon 0 \leq i \leq n-1\} \cup \{\sum_{i=0}^{n-1}(\gamma_{i}u_{i})\})$. We have the following $k$-linearly independent subsets: $\{v_{i} \vcentcolon 0 \leq i \leq n-1\} \subseteq aE_{(t+1_{k}),n}$, $\{v_{i} \vcentcolon 1 \leq i \leq \nolinebreak n-1\} \cup \{\sum_{i = 0}^{n - 1}(\alpha_{i}v_{i})\} \subseteq bE_{(t+1_{k}),n}$ and $\{v_{i} + v_{i+1} \vcentcolon 0 \leq i \leq n-2\} \subseteq (a+b)E_{(t+1_{k}),n}$.
	\begin{spacing}{1.3}
		It remains to show that $(a+b)\sum_{i=0}^{n-1}(\gamma_{i}u_{i}) = 0_{E_{(t+1_{k}),n}}$. Observe that $\gamma_{i-1} + \alpha_{i} = \gamma_{i}$ for $1 \leq i \leq n-1$ and $\gamma_{n-1} = \sum_{j=0}^{n-1}(\alpha_{j}) = 1_{k}$, because $0_{k} = (1_{k}+1_{k})^{n} = 1_{k} + \sum_{i=0}^{n-1}(\alpha_{i})$. Clearly $a\sum_{i=0}^{n-1}(\gamma_{i}u_{i}) = \sum_{i=0}^{n-1}(\gamma_{i}v_{i})$, so we must show that $b\sum_{i=0}^{n-1}(\gamma_{i}u_{i}) = \sum_{i=0}^{n-1}(\gamma_{i}v_{i})$. Now, $b\sum_{i=0}^{n-1}(\gamma_{i}u_{i}) = \sum_{i=0}^{n-2}(\gamma_{i}v_{i+1}) + \gamma_{n-1}\sum_{i = 0}^{n - 1}(\alpha_{i} v_{i}) = \sum_{i=1}^{n-1}(\gamma_{i-1}v_{i}) + \sum_{i = 0}^{n - 1}(\alpha_{i} v_{i}) = \sum_{i=1}^{n-1}((\gamma_{i-1} + \alpha_{i})v_{i}) + \alpha_{0} v_{0} = \sum_{i=1}^{n-1}(\gamma_{i}v_{i}) + \gamma_{0}v_{0} = \sum_{i=0}^{n-1}(\gamma_{i}v_{i})$, so we are done.
		\qedhere
	\end{spacing}
	\vspace{-0.2\baselineskip}
\end{proof}

\begin{prop}
\label{prop:efker}
	Let $n \geq 1$ and $f \in k[t]$ a monic irreducible polynomial. Note that $f^{n} = t^{m} + \sum_{i = 0}^{m - 1}(\alpha_{i} t^{i})$ for some $\alpha_{i} \in k$, where $m \defeq n \deg(f)$. Suppose that $f \notin \{t,\, t+1_{k}\}$. For $E_{f,n}$ we have:
	\[
	\ker(a) + \ker(a+b) + \ker(b) = \vspan_{k}\{v_{i} \vcentcolon 0 \leq i \leq m-1\}
	\]
\end{prop}

\begin{proof}
	\[\begin{tikzcd}[cramped,sep=tiny]
		& {\overset{u_{0}}\bullet} && {\overset{u_{1}}\bullet} & {} && {\overset{u_{m - 1}}\bullet} \\
		{\underset{v_{0}}\bullet} && {\underset{v_{1}}\bullet} && {} & {\underset{v_{m - 1}}\bullet} && {\underset{\sum_{i = 0}^{m - 1}(\alpha_{i} v_{i})}{\circ}}
		\arrow[no head, from=1-2, to=2-1]
		\arrow[no head, from=1-2, to=2-3]
		\arrow[no head, from=1-4, to=2-3]
		\arrow[shorten >=10pt, no head, from=1-4, to=2-5]
		\arrow["\ldots"{description, pos=0.4}, draw=none, from=1-5, to=2-5]
		\arrow[no head, from=1-7, to=2-6]
		\arrow[dashed, no head, from=1-7, to=2-8]
		\arrow[shorten >=10pt, no head, from=2-6, to=1-5]
	\end{tikzcd}\]
	We will in fact show that $\ker(a) = \ker(a+b) = \ker(b) = \vspan_{k}\{v_{i} \vcentcolon 0 \leq i \leq m-1\}$. Since $\{v_{i} \vcentcolon 0 \leq i \leq m-1\} \subseteq \ker(a) \cap \ker(a+b) \cap \ker(b)$ is a $k$-linearly independent subset, it suffices to show that the $k$-dimensions of $aE_{f,n}$, $(a+b)E_{f,n}$ and $bE_{f,n}$ are at least $m$, by the Rank-Nullity Theorem. We have the following $k$-linearly independent subsets: $\{v_{i} \vcentcolon 0 \leq i \leq m-1\} \subseteq aE_{f,n}$, $\{v_{i} \vcentcolon 1 \leq i \leq m-1\} \cup \{\sum_{i = 0}^{m - 1}(\alpha_{i}v_{i})\} \subseteq bE_{f,n}$ and $\{v_{i} + v_{i+1} \vcentcolon 0 \leq i \leq m-2\} \cup \{v_{m-1} + \sum_{i = 0}^{m - 1}(\alpha_{i} v_{i})\} \subseteq (a+b)E_{f,n}$. This is straightforward to verify for the first two (note that $\alpha_{0} \neq 0_{k}$ since $f \neq t$ is irreducible) but the third is a little more involved. If $\sum_{i=0}^{m-2}(\lambda_{i}(v_{i} + v_{i+1})) + \lambda_{m-1}(v_{m-1} + \sum_{i = 0}^{m - 1}(\alpha_{i} v_{i})) = 0_{E_{f,n}}$ for some $\lambda_{i} \in k$, then $\sum_{i=0}^{m-2}(\lambda_{i}v_{i}) + \sum_{i=1}^{m-1}(\lambda_{i-1}v_{i}) + \lambda_{m-1}v_{m-1} = \lambda_{m-1}\sum_{i = 0}^{m - 1}(\alpha_{i} v_{i})$ and therefore $\lambda_{0}v_{0} + \sum_{i=1}^{m-1}((\lambda_{i}+\lambda_{i-1})v_{i}) = \lambda_{m-1}\sum_{i = 0}^{m - 1}(\alpha_{i} v_{i})$. Comparing coefficients gives $\lambda_{0} = \lambda_{m-1}\alpha_{0}$ and $\lambda_{i} = \lambda_{i-1}+\lambda_{m-1}\alpha_{i}$ for $1 \leq i \leq m-1$. It follows that $\lambda_{i} = \lambda_{m-1}\sum_{j=0}^{i}(\alpha_{j})$ for $0 \leq i \leq m-1$. In particular, we have $\lambda_{m-1} = \lambda_{m-1}\sum_{j=0}^{m-1}(\alpha_{j})$, so either $\lambda_{m-1} = 0_{k}$ or $\sum_{j=0}^{m-1}(\alpha_{j}) = 1_{k}$. In the latter case, we have $f^{n}(1_{k}) = 1_{k}^{m} + \sum_{i = 0}^{m - 1}(\alpha_{i} 1_{k}^{i}) = 0_{k}$, so $(t+1_{k})$ is a factor of $f^{n}$ and also $f$. Since $f$ is monic and irreducible, we must have $f = t+1_{k}$, but this contradicts our assumption. Therefore $\lambda_{m-1} = 0_{k}$ and $\lambda_{i} = \lambda_{m-1}\sum_{j=0}^{i}(\alpha_{j}) = 0_{k}$ for $0 \leq i \leq m-1$, so we are done.
\end{proof}

\begin{thm}
\label{thm:pfsub}
	Let $P$ be a projective-free permutation module and $M$ an indecomposable submodule. Then $M$ is isomorphic to one of $k$, $E_{t}$, $E_{t+1_{k}}$, $E_{\infty}$, $M_{3}$ or $M_{5}$. It follows that any submodule of $P$ is isomorphic to a direct sum of the above indecomposables.
\end{thm}

\begin{proof}
	The inclusion $\iota \vcentcolon M \inj{} P$ yields a surjection $\iota^{*} \vcentcolon P^{*} \sur{} M^{*}$ as in Recollection \ref{rec:dual}. Now $P^{*} \cong P$ is a projective free permutation module, so $\ker(a) + \ker(a+b) + \ker(b) = M^{*}$ as in Remark \ref{rem:pfhom}. Moreover, we know that $M^{*}$ is indecomposable since $M$ is indecomposable. Propositions \ref{prop:mker} to \ref{prop:efker} tell us that this can only occur when $M^{*}$ is isomorphic to one of $k$, $E_{t}$, $E_{t+1_{k}}$, $E_{\infty}$, $W_{3}$ or $W_{5}$, by counting $k$-dimensions. This means that $M \cong M^{**}$ is isomorphic to one of $k$, $E_{t}$, $E_{t+1_{k}}$, $E_{\infty}$, $M_{3}$ or $M_{5}$, using Lemma \ref{lem:dual}.
\end{proof}

\begin{rem}
	In fact, each of $k$, $E_{t}$, $E_{t+1_{k}}$, $E_{\infty}$, $M_{3}$ and $M_{5}$ can be realised as a submodule of a projective-free permutation module. This is immediate for $k$, $E_{t}$, $E_{t+1_{k}}$ and $E_{\infty}$.\linebreak For $M_{3}$ and $M_{5}$, the inclusions are described below, with the details left to the reader:
	%
	%
	%
	\[\begin{array}{ccc}
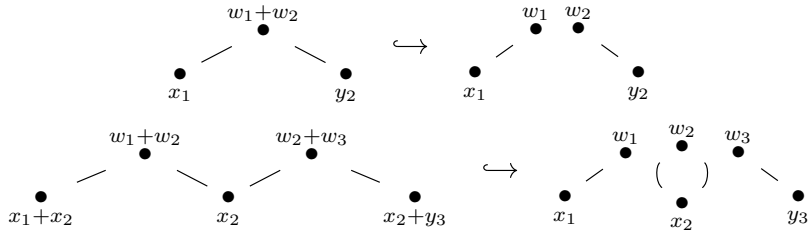
		\begin{tikzcd}[cramped,sep=tiny]
			& {\overset{w_{1}+w_{2}}\bullet} \\
			{\underset{x_{1}}\bullet} && {\underset{y_{2}}\bullet}
			\arrow[no head, from=1-2, to=2-1]
			\arrow[no head, from=1-2, to=2-3]
		\end{tikzcd}
		&
		\inj{}
		&
		\begin{tikzcd}[cramped,sep=tiny]
			& {\overset{w_{1}}\bullet} \\
			{\underset{x_{1}}\bullet}
			\arrow[no head, from=1-2, to=2-1]
		\end{tikzcd}
		\hspace{2mm}
		\begin{tikzcd}[cramped,sep=tiny]
			{\overset{w_{2}}\bullet} \\
			& {\underset{y_{2}}\bullet}
			\arrow[no head, from=1-1, to=2-2]
		\end{tikzcd}
	\end{array}\]
	%
	%
	%
	%
	\[\begin{array}{ccc}
		\begin{tikzcd}[cramped,sep=tiny]
			& {\overset{w_{1}+w_{2}}\bullet} && {\overset{w_{2}+w_{3}}\bullet} \\
			{\underset{x_{1}+x_{2}}\bullet} && {\underset{x_{2}}\bullet} && {\underset{x_{2}+y_{3}}\bullet}
			\arrow[no head, from=1-2, to=2-1]
			\arrow[no head, from=1-2, to=2-3]
			\arrow[no head, from=1-4, to=2-3]
			\arrow[no head, from=1-4, to=2-5]
		\end{tikzcd}
		&
		\inj{}
		&
		\begin{tikzcd}[cramped,sep=tiny]
			& {\overset{w_{1}}\bullet} \\
			{\underset{x_{1}}\bullet}
			\arrow[no head, from=1-2, to=2-1]
		\end{tikzcd}
		\hspace{2mm}
		\begin{tikzcd}[cramped,sep=small]
			{\overset{w_{2}}\bullet} \\
			{\underset{x_{2}}\bullet}
			\arrow[curve={height=-10pt}, no head, from=1-1, to=2-1]
			\arrow[curve={height=10pt}, no head, from=1-1, to=2-1]
		\end{tikzcd}
		\hspace{2mm}
		\begin{tikzcd}[cramped,sep=tiny]
			{\overset{w_{3}}\bullet} \\
			& {\underset{y_{3}}\bullet}
			\arrow[no head, from=1-1, to=2-2]
		\end{tikzcd}
	\end{array}\]
\end{rem}

\subsection{Quotients}
\label{sec:quo}
\mbox{}\\
The results from Section \ref{sec:proj} are also useful in understanding when $E_{t}$, $E_{t+1_{k}}$ and $E_{\infty}$ can be realised as submodules of a given indecomposable. The structure of the corresponding quotient module will be important to us later on.

\begin{rem}
\label{rem:Esub}
	Let $n \geq 1$ and $f \in k[t]$ a monic irreducible polynomial with $f \notin \{t,\, t+1_{k}\}$. Since the $kV_{4}$-action on $\ker(a) = \ker(a+b) = \ker(b)$ is trivial for both of $M_{2n+1}$ and $E_{f,n}$, it follows that $E_{t}$, $E_{t+1_{k}}$ and $E_{\infty}$ cannot occur as submodules of either one. When $f \defeq t$, the $kV_{4}$-action on $\ker(a) = \ker(a+b) \subseteq E_{t,n}$ is trivial, so $E_{\infty}$ and $E_{t+1_{k}}$ cannot occur as submodules of $E_{t,n}$. Similarly, $E_{t}$ and $E_{t+1_{k}}$ cannot occur as submodules of $E_{\infty,n}$, and finally, $E_{\infty}$ and $E_{t}$ cannot occur as submodules of $E_{(t+1_{k}),n}$.
\end{rem}

\begin{lem}
\label{lem:cokerE}
	Let $n \geq 2$. If $\phi \vcentcolon E_{t} \inj{} E_{t,n}$ is an injective morphism, then we have $\coker(\phi) \cong E_{t,n-1}$. Similarly, if $\phi' \vcentcolon E_{\infty} \inj{} E_{\infty,n}$ and $\phi'' \vcentcolon E_{t+1_{k}} \inj{} E_{(t+1_{k}),n}$ are both injective morphisms, then $\coker(\phi') \cong E_{\infty,n-1}$ and $\coker(\phi'') \cong E_{(t+1_{k}),n-1}$.
\end{lem}

\begin{proof}
	We have $\phi(E_{t}) \subseteq \ker(b) = \vspan_{k}(\{v_{i} \vcentcolon 0 \leq i \leq n-1\} \cup \{u_{n-1}\}) \subseteq E_{t,n}$, so $\phi(E_{t})$ has a $k$-basis of the form $\{\lambda u_{n-1} + \sum_{i=0}^{n-1}(\mu_{i} v_{i}),\, \lambda v_{n-1}\}$ for some $\lambda, \mu_{i} \in k$ with $\lambda \neq 0_{k}$.\linebreak It follows that a $k$-basis for $\coker(\phi)$ is given by $\left\{u_{i} \vcentcolon 0 \leq i \leq n-2\right\} \cup \left\{v_{i} \vcentcolon 0 \leq i \leq n-2\right\}$, since this is a spanning set of size $2n-2$. By looking at the $kV_{4}$-action on this $k$-basis, it is clear that $\coker(\phi) \cong E_{t,n-1}$. Similarly, we see that $\phi'(E_{\infty})$ has a $k$-basis of the form $\{\lambda u_{0} + \sum_{i=0}^{n-1}(\mu_{i} v_{i}),\, \lambda v_{0}\}$ for some $\lambda, \mu_{i} \in k$ with $\lambda \neq 0_{k}$. It follows that a $k$-basis for $\coker(\phi')$ is given by $\left\{u_{i} \vcentcolon 1 \leq i \leq n-1\right\} \cup \left\{v_{i} \vcentcolon 1 \leq i \leq n-1\right\}$, and indeed we have $\coker(\phi') \cong E_{\infty,n-1}$. When checking the $kV_{4}$-actions, the main thing to verify is that $bu_{n-2} = v_{n-1} = 0$ in the first diagram and $au_{1} = v_{0} = 0$ in the second diagram:
	
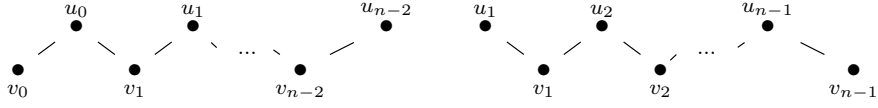
\[\begin{tikzcd}[cramped,sep=tiny]
		& {\overset{u_{0}}\bullet} && {\overset{u_{1}}\bullet} & {} && {\overset{u_{n - 2}}\bullet} \\
		{\underset{v_{0}}\bullet} && {\underset{v_{1}}\bullet} && {} & {\underset{v_{n - 2}}\bullet}
		\arrow[no head, from=1-2, to=2-1]
		\arrow[no head, from=1-2, to=2-3]
		\arrow[no head, from=1-4, to=2-3]
		\arrow[shorten >=11pt, no head, from=1-4, to=2-5]
		\arrow["\ldots"{description, pos=0.4}, shift left, draw=none, from=1-5, to=2-5]
		\arrow[no head, from=1-7, to=2-6]
		\arrow[shorten >=11pt, no head, from=2-6, to=1-5]
	\end{tikzcd}
	\hspace{8mm}
	\begin{tikzcd}[cramped,sep=tiny]
		{\overset{u_{1}}\bullet} && {\overset{u_{2}}\bullet} && {} & {\overset{u_{n - 1}}\bullet} \\
		& {\underset{v_{1}}\bullet} && {\underset{v_{2}}\bullet} & {} && {\underset{v_{n - 1}}\bullet}
		\arrow[no head, from=1-1, to=2-2]
		\arrow[no head, from=1-3, to=2-2]
		\arrow[no head, from=1-3, to=2-4]
		\arrow["\ldots"{description, pos=0.4}, draw=none, from=1-5, to=2-5]
		\arrow[shorten >=13pt, no head, from=1-6, to=2-5]
		\arrow[no head, from=1-6, to=2-7]
		\arrow[shorten >=14pt, no head, from=2-4, to=1-5]
	\end{tikzcd}\]
	Finally, $\phi''(E_{t+1_{k}})$ has a $k$-basis of the form $\{\lambda \sum_{i=0}^{n-1}(\gamma_{i}u_{i}) + \sum_{i=0}^{n-1}(\mu_{i}v_{i}),\, \lambda \sum_{i=0}^{n-1}(\gamma_{i}v_{i})\}$ for some $\lambda, \mu_{i} \in k$ with $\lambda \neq 0_{k}$. Recall that $\gamma_{i} \defeq \sum_{j=0}^{i}(\alpha_{j})$ with $\alpha_{j} \in k$ such that $(t+1_{k})^{n} = t^{n} + \sum_{j=0}^{n-1}(\alpha_{j}t^{j})$, as in Proposition \ref{prop:et1ker}. We also saw that $\gamma_{n-1} = 1_{k}$, which means that $u_{n-1}$ and $v_{n-1}$ can be expressed as a $k$-linear combination of the remaining basis elements when considering their images in $\coker(\phi'')$. It follows that a $k$-basis for $\coker(\phi'')$ is given by $\left\{u_{i} \vcentcolon 0 \leq i \leq n-2\right\} \cup \left\{v_{i} \vcentcolon 0 \leq i \leq n-2\right\}$. Note that $bu_{n-2} = v_{n-1} = \sum_{i=0}^{n-2}(\gamma_{i}v_{i})$ so $\coker(\phi'')$ is described by the following diagram:
	
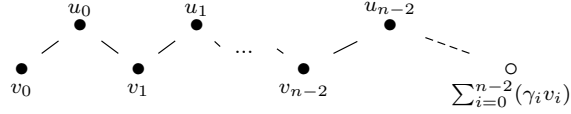
\[\begin{tikzcd}[cramped,sep=tiny]
		& {\overset{u_{0}}\bullet} && {\overset{u_{1}}\bullet} & {} && {\overset{u_{n - 2}}\bullet} \\
		{\underset{v_{0}}\bullet} && {\underset{v_{1}}\bullet} && {} & {\underset{v_{n - 2}}\bullet} && {\underset{\sum_{i = 0}^{n - 2}(\gamma_{i} v_{i})}{\circ}}
		\arrow[no head, from=1-2, to=2-1]
		\arrow[no head, from=1-2, to=2-3]
		\arrow[no head, from=1-4, to=2-3]
		\arrow[shorten >=11pt, no head, from=1-4, to=2-5]
		\arrow["\ldots"{description, pos=0.4}, draw=none, from=1-5, to=2-5]
		\arrow[no head, from=1-7, to=2-6]
		\arrow[dashed, no head, from=1-7, to=2-8]
		\arrow[shorten >=11pt, no head, from=2-6, to=1-5]
	\end{tikzcd}\]
	\begin{spacing}{1.3}
		To conclude that this is isomorphic to $E_{(t+1_{k}),n-1}$, we must show that $t^{n-1}+\sum_{i=0}^{n-2}(\gamma_{i}t^{i}) = (t+1_{k})^{n-1}$. We have $(t+1_{k})(t^{n-1}+\sum_{i=0}^{n-2}(\gamma_{i}t^{i})) = t^{n} + \sum_{i=0}^{n-2}(\gamma_{i}t^{i+1}) + t^{n-1} + \sum_{i=0}^{n-2}(\gamma_{i}t^{i}) = t^{n} + \sum_{i=1}^{n-1}(\gamma_{i-1}t^{i}) + \sum_{i=0}^{n-1}(\gamma_{i}t^{i}) = t^{n} + \sum_{i=1}^{n-1}\big((\gamma_{i-1}+\gamma_{i})t^{i}\big) + \gamma_{0} = t^{n} + \sum_{i=0}^{n-1}(\alpha_{i}t^{i}) = (t+1_{k})^{n}$, so the result follows.
		\qedhere
	\end{spacing}
	\vspace{-0.2\baselineskip}
\end{proof}

\begin{lem}
\label{lem:cokerW}
	Let $n \geq 2$. If $\phi \vcentcolon E_{t} \inj{} W_{2n+1}$ is an injective morphism, then we have $\coker(\phi) \cong W_{2n-1}$. Similarly, if $\phi' \vcentcolon E_{\infty} \inj{} W_{2n+1}$ and $\phi'' \vcentcolon E_{t+1_{k}} \inj{} W_{2n+1}$ are both injective morphisms, then $\coker(\phi') \cong \coker(\phi'') \cong W_{2n-1}$.
\end{lem}

\vspace{-0.5\baselineskip}

\begin{proof}
	We have $\phi(E_{t}) \subseteq \ker(b) = \vspan_{k}(\{v_{i} \vcentcolon 1 \leq i \leq n\} \cup \{u_{n}\}) \subseteq W_{2n+1}$, so $\phi(E_{t})$ has a $k$-basis of the form $\{\lambda u_{n} + \sum_{i=1}^{n}(\mu_{i}v_{i}),\, \lambda v_{n}\}$ for some $\lambda, \mu_{i} \in k$ with $\lambda \neq 0_{k}$.\linebreak It follows that a $k$-basis for $\coker(\phi)$ is given by $\left\{u_{i} \vcentcolon 0 \leq i \leq n-1\right\} \cup \left\{v_{i} \vcentcolon 1 \leq i \leq n-1\right\}$, since this is a spanning set of size $2n-1$. By looking at the $kV_{4}$-action on this $k$-basis, it is clear that $\coker(\phi) \cong W_{2n-1}$. Similarly, we see that $\phi'(E_{\infty})$ has a $k$-basis of the form $\{\lambda u_{0} + \sum_{i=1}^{n}(\mu_{i} v_{i}),\, \lambda v_{1}\}$ for some $\lambda, \mu_{i} \in k$ with $\lambda \neq 0_{k}$. It follows that a $k$-basis \linebreak for $\coker(\phi')$ is given by $\left\{u_{i} \vcentcolon 1 \leq i \leq n\right\} \cup \left\{v_{i} \vcentcolon 2 \leq i \leq n\right\}$, and indeed $\coker(\phi') \cong W_{2n-1}$. When checking the $kV_{4}$-actions, the main thing to verify is that $bu_{n-1} = v_{n} = 0$ in the first diagram and $au_{1} = v_{1} = 0$ in the second diagram:
	
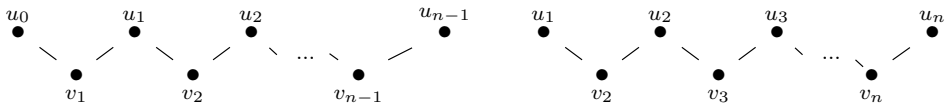
\[\begin{tikzcd}[cramped,sep=tiny]
		{\overset{u_{0}}\bullet} && {\overset{u_{1}}\bullet} && {\overset{u_{2}}\bullet} & {} && {\overset{u_{n-1}}\bullet} \\
		& {\underset{v_{1}}\bullet} && {\underset{v_{2}}\bullet} && {} & {\underset{v_{n-1}}\bullet}
		\arrow[no head, from=1-1, to=2-2]
		\arrow[no head, from=1-3, to=2-2]
		\arrow[no head, from=1-3, to=2-4]
		\arrow[shorten >=10pt, no head, from=1-5, to=2-6]
		\arrow["\ldots"{description, pos=0.4}, shift left, draw=none, from=1-6, to=2-6]
		\arrow[no head, from=1-8, to=2-7]
		\arrow[no head, from=2-4, to=1-5]
		\arrow[shorten >=10pt, no head, from=2-7, to=1-6]
	\end{tikzcd}
	\hspace{8mm}
	\begin{tikzcd}[cramped,sep=tiny]
		{\overset{u_{1}}\bullet} && {\overset{u_{2}}\bullet} && {\overset{u_{3}}\bullet} & {} && {\overset{u_{n}}\bullet} \\
		& {\underset{v_{2}}\bullet} && {\underset{v_{3}}\bullet} && {} & {\underset{v_{n}}\bullet}
		\arrow[no head, from=1-1, to=2-2]
		\arrow[no head, from=1-3, to=2-2]
		\arrow[no head, from=1-3, to=2-4]
		\arrow[shorten >=10pt, no head, from=1-5, to=2-6]
		\arrow["\ldots"{description, pos=0.4}, shift left, draw=none, from=1-6, to=2-6]
		\arrow[no head, from=1-8, to=2-7]
		\arrow[no head, from=2-4, to=1-5]
		\arrow[shorten >=15pt, no head, from=2-7, to=1-6]
	\end{tikzcd}\]
	Finally, $\phi''(E_{t+1_{k}})$ has a $k$-basis of the form $\{\lambda \sum_{i=0}^{n}(u_{i}) + \sum_{i=1}^{n}(\mu_{i} v_{i}),\, \lambda \sum_{i=1}^{n}(v_{i})\}$ for some $\lambda, \mu_{i} \in k$ with $\lambda \neq 0_{k}$. Note that we can express $u_{n}$ and $v_{n}$ as a $k$-linear combination of the remaining basis elements when considering their images in $\coker(\phi'')$. It follows that a $k$-basis for $\coker(\phi'')$ is given by $\left\{u_{i} \vcentcolon 0 \leq i \leq n-1\right\} \cup \left\{v_{i} \vcentcolon 1 \leq i \leq n-1\right\}$. To see that $\coker(\phi'') \cong W_{2n-1}$, we change the $k$-basis to $\left\{\gamma_{i} \vcentcolon 0 \leq i \leq n-1\right\} \cup \left\{\delta_{i} \vcentcolon 1 \leq i \leq n-1\right\}$ where $\gamma_{i} \defeq \sum_{j=0}^{i}(u_{j})$ for $0 \leq i \leq n-1$ and $\delta_{i} \defeq \sum_{j=1}^{i}(v_{j})$ for $1 \leq i \leq n-1$. We claim that the $kV_{4}$-action on this $k$-basis is described by the following diagram:
	
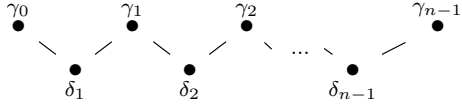
\[\begin{tikzcd}[cramped,sep=tiny]
		{\overset{\gamma_{0}}\bullet} && {\overset{\gamma_{1}}\bullet} && {\overset{\gamma_{2}}\bullet} & {} && {\overset{\gamma_{n-1}}\bullet} \\
		& {\underset{\delta_{1}}\bullet} && {\underset{\delta_{2}}\bullet} && {} & {\underset{\delta_{n-1}}\bullet}
		\arrow[no head, from=1-1, to=2-2]
		\arrow[no head, from=1-3, to=2-2]
		\arrow[no head, from=1-3, to=2-4]
		\arrow[shorten >=10pt, no head, from=1-5, to=2-6]
		\arrow["\ldots"{description, pos=0.4}, shift left, draw=none, from=1-6, to=2-6]
		\arrow[no head, from=1-8, to=2-7]
		\arrow[no head, from=2-4, to=1-5]
		\arrow[shorten >=10pt, no head, from=2-7, to=1-6]
	\end{tikzcd}\]
	\begin{spacing}{1.3}
		Firstly, note that $a\delta_{i} = b\delta_{i} = 0$ for $1 \leq i \leq n-1$. For $1 \leq i \leq n-1$, we also have $a\gamma_{i} = \sum_{j=0}^{i}(au_{j}) = \sum_{j=1}^{i}(v_{j}) = \delta_{i}$ and for $0 \leq i \leq n-2$ we have $b\gamma_{i} = \sum_{j=0}^{i}(bu_{j}) = \sum_{j=0}^{i}(v_{j+1}) = \sum_{j=1}^{i+1}(v_{j}) = \delta_{i+1}$. Finally, $a\gamma_{0} = au_{0} = 0$ and $b\gamma_{n-1} = \sum_{j=0}^{n-1}(bu_{j}) = \sum_{j=0}^{n-1}(v_{j+1}) = \sum_{j=1}^{n}(v_{j}) = 0$. Therefore $\coker(\phi'')$ is indeed isomorphic to $W_{2n-1}$.
		\qedhere
	\end{spacing}
	\vspace{-0.2\baselineskip}
\end{proof}

\subsection{Dimension Two}
\label{sec:dim2}
\mbox{}\\
We are now ready to show that the indecomposables not listed in Theorems \ref{thm:dim0} and \ref{thm:dim1} have permutation dimension equal to two. We know that their permutation dimension must be either one or two, so it suffices to show that it cannot be equal to one.

\begin{lem}
\label{lem:essPerm}
	Let $M$ be a $kV_{4}$-module with $\ppdim(M) = 1$. Then there exists an essential surjection $\psi \vcentcolon P' \sur{} M$ such that $P'$ is a permutation module and $\ker(\psi)$ is a submodule of a projective-free permutation module.
\end{lem}

\begin{proof}
	First, note that the head of each indecomposable permutation module is isomorphic to the trivial module $k$. It follows that the number of indecomposable summands in a decomposition of a permutation module is equal to the $k$-dimension of its head.
	
	Since $\ppdim(M) = 1$, there exists a short exact sequence $\ker(\phi) \inj{} P \sur{\phi} M$ such that both $P$ and $\ker(\phi)$ are permutation modules. We may assume, without loss of generality, that $\ker(\phi)$ is a projective-free permutation module. Indeed, projective modules are also injective, so any projective summands in $\ker(\phi)$ can be `split off' from the short exact sequence since their image in $P$ is also direct summand. A permutation module with projective summands removed is of course still a permutation module.

	Let $P = \bigoplus_{i = 1}^{n}(P_{i})$ be a decomposition of $P$ into indecomposable permutation modules. Since each $P_{i}$ is principally generated as a $kV_{4}$-module, we may choose a generating set $\{w_{i} \vcentcolon 1 \leq i \leq n\}$ for $P$ such that each $P_{i}$ is generated by $w_{i}$. Now $\{\phi(w_{i}) \vcentcolon 1 \leq i \leq n\}$ is a generating set for $M$, so $\{\phi(w_{i}) + \rad(M) \vcentcolon 1 \leq i \leq n\}$ is a generating set for $\hd(M)$. Since $a$ and $b$ act trivially on $\hd(M)$, this is in fact a generating set for $\hd(M)$ as a $k$-vector space, so we may reduce it to a $k$-basis denoted by $\{\phi(w_{i}) + \rad(M) \vcentcolon 1 \leq i \leq n'\}$ without loss of generality, where $n' \defeq \dim_{k}(\hd(M))$.
	
	Let $P' \defeq \bigoplus_{i = 1}^{n'}(P_{i})$. The image of $\phi(P') \subseteq M$ under the quotient map $M \sur{\pi} \hd(M)$ contains $\{\phi(w_{i}) + \rad(M) \vcentcolon 1 \leq i \leq n'\}$, hence it is all of $\hd(M)$. Proposition \ref{prop:nak} tells us that $\pi$ is an essential surjection, so $\phi(P') = M$. Define $\psi \vcentcolon P' \sur{} M$ as the restriction of $\phi$ to $P'$. Note that $\dim_{k}(\hd(P')) = n' = \dim_{k}(\hd(M))$, so $\psi$ is an essential surjection by Proposition \ref{prop:esshd}. Finally, $\ker(\psi) = \ker(\phi|_{P'})$ is a submodule of $\ker(\phi)$, which is a projective-free permutation module.
\end{proof}

\begin{lem}
\label{lem:snake}
	Let $\phi \vcentcolon M' \sur{} M$ be an essential surjection of $kV_{4}$-modules. Then there exists a short exact sequence of the form:
	$\;$
	$\begin{tikzcd}[cramped,sep=small]
		0 & {\Omega(M')} & {\Omega(M)} & {\ker(\phi)} & 0
		\arrow[from=1-1, to=1-2]
		\arrow[from=1-2, to=1-3]
		\arrow[from=1-3, to=1-4]
		\arrow[from=1-4, to=1-5]
	\end{tikzcd}$
\end{lem}

\begin{proof}
	Let $\psi \vcentcolon P \sur{} M'$ be a projective cover of $M'$. Then $\phi \circ \psi \vcentcolon P \sur{} M$ is a projective cover of $M$. Indeed, both $\psi$ and $\phi$ are essential, so if $P' \subsetneq P$ is a proper submodule of $P$ then $\psi(P') \subsetneq M'$ and so $\phi \circ \psi(P') \subsetneq M$, hence $\phi \circ \psi$ is essential. Consider the following commutative diagram:
	\vspace{-0.7\baselineskip}
	\[\begin{tikzcd}[cramped,sep=scriptsize]
		&&& {\ker(\phi)} \\
		0 & {\Omega(M')} & P & {M'} & 0 \\
		0 & {\Omega(M)} & P & M & 0 \\
		& {\coker(\iota)}
		\arrow[hook, from=1-4, to=2-4]
		\arrow[from=2-1, to=2-2]
		\arrow[hook, from=2-2, to=2-3]
		\arrow["\iota", hook, from=2-2, to=3-2]
		\arrow["\psi", two heads, from=2-3, to=2-4]
		\arrow["{\big|\big|}"{description}, draw=none, from=2-3, to=3-3]
		\arrow[from=2-4, to=2-5]
		\arrow["\phi", two heads, from=2-4, to=3-4]
		\arrow[from=3-1, to=3-2]
		\arrow[hook, from=3-2, to=3-3]
		\arrow[two heads, from=3-2, to=4-2]
		\arrow["{\phi \circ \psi}"', two heads, from=3-3, to=3-4]
		\arrow[from=3-4, to=3-5]
	\end{tikzcd}\]
	The rows are exact by definition of the Heller operator as the kernel of a projective cover. Note that $\phi \circ \psi(\Omega(M')) = \phi(\{0_{M'}\}) = \{0_{M}\}$, so $\Omega(M') \subseteq \Omega(M)$ as indicated by the inclusion $\iota$ in the diagram. The Snake Lemma gives us an isomorphism $\ker(\phi) \cong \coker(\iota)$, therefore we have a short exact sequence:
	$\begin{tikzcd}[cramped,sep=small]
		0 & {\Omega(M')} & {\Omega(M)} & {\ker(\phi)} & 0
		\arrow[from=1-1, to=1-2]
		\arrow[from=1-2, to=1-3]
		\arrow[from=1-3, to=1-4]
		\arrow[from=1-4, to=1-5]
	\end{tikzcd}$
\end{proof}

\begin{cor}
\label{cor:SES}
	Let $M$ be an indecomposable $kV_{4}$-module with $\ppdim(M) = 1$. Then there exists a short exact sequence of the form
	$
	\;
	\begin{tikzcd}[cramped,sep=small]
		0 & P & {\Omega(M)} & N & 0
		\arrow[from=1-1, to=1-2]
		\arrow[from=1-2, to=1-3]
		\arrow[from=1-3, to=1-4]
		\arrow[from=1-4, to=1-5]
	\end{tikzcd}
	\;
	$
	with $N$ a submodule of a projective-free permutation module and $P$ a permutation module with indecomposable summands among $E_{t}$, $E_{t+1_{k}}$ and $E_{\infty}$.
\end{cor}

\begin{proof}
	In the notation of Lemma \ref{lem:essPerm}, we obtain a short exact sequence:
	\[\begin{tikzcd}[cramped,sep=small]
		0 & {\Omega(P')} & {\Omega(M)} & {\ker(\psi)} & 0
		\arrow[from=1-1, to=1-2]
		\arrow[from=1-2, to=1-3]
		\arrow[from=1-3, to=1-4]
		\arrow[from=1-4, to=1-5]
	\end{tikzcd}\]
	by applying Lemma \ref{lem:snake}. We take $P \defeq \Omega(P')$ and $N \defeq \ker(\psi)$. It remains to show that the indecomposable summands of $P$ are among $E_{t}$, $E_{t+1_{k}}$ and $E_{\infty}$. We first show that $k$ cannot be a summand of $P'$. Since $\psi$ is an essential surjection, we know that $\bar{\psi} \vcentcolon \hd(P') \to \hd(M)$ is an isomorphism as in Proposition \ref{prop:esshd}. If $k$ is a summand of $P'$, this restricts to an isomorphism $\bar{\psi} \vcentcolon k \to \bar{\psi}(k)$, since $\hd(k) = k$. Now, it is straightforward to check that any non-trivial indecomposable $M$ satisfies $\ker(a) \cap \ker(b) \subseteq \rad(M)$, and $M$ is certainly non-trivial because $\ppdim(M)=1$. Since $\psi(k) \subseteq \ker(a) \cap \ker(b)$, it follows that $k$ vanishes under the map $\bar{\psi}$, a contradiction. We know that $P'$ is a permutation module, so its indecomposable summands must be among $kV_{4}$, $E_{t}$, $E_{t+1_{k}}$ and $E_{\infty}$. It follows that the indecomposable summands of $P \defeq \Omega(P')$ are among $E_{t}$, $E_{t+1_{k}}$ and $E_{\infty}$. To be explicit, $E_{t}$, $E_{t+1_{k}}$ and $E_{\infty}$ are fixed by $\Omega(-)$, whereas $\Omega(kV_{4}) = 0$.
\end{proof}

\begin{thm}
\label{thm:ppdim2M}
	Let $n \geq 4$. Then $\ppdim(M_{2n+1}) = 2$.
\end{thm}

\begin{proof}
	Suppose that $\ppdim(M_{2n+1}) = 1$ and consider the short exact sequence:
	\[\begin{tikzcd}[cramped,sep=small]
		0 & P & {M_{2n-1}} & N & 0
		\arrow[from=1-1, to=1-2]
		\arrow[from=1-2, to=1-3]
		\arrow[from=1-3, to=1-4]
		\arrow[from=1-4, to=1-5]
	\end{tikzcd}\]
	from Corollary \ref{cor:SES}, noting that $\Omega(M_{2n+1}) \cong M_{2n-1}$. As in Remark \ref{rem:Esub}, $M_{2n-1}$ does not admit $E_{t}$, $E_{t+1_{k}}$ or $E_{\infty}$ as submodules, so we must have $P = 0$. Then $M_{2n-1} \cong N$ is a submodule of a projective-free permutation module, but this contradicts Theorem \ref{thm:pfsub}. We conclude that $\ppdim(M_{2n+1}) \neq 1$ and so $\ppdim(M_{2n+1}) = 2$ by Proposition \ref{prop:ppdimM}.
\end{proof}

\begin{thm}
	Let $n \geq 1$ and $f \in k[t]$ a monic irreducible polynomial with $f \neq t$ and $f \neq t+1_{k}$. Then $\ppdim(E_{f,n}) = 2$.
\end{thm}

\begin{proof}
	Suppose that $\ppdim(E_{f,n}) = 1$ and consider the short exact sequence:
	\[\begin{tikzcd}[cramped,sep=small]
		0 & P & {E_{f,n}} & N & 0
		\arrow[from=1-1, to=1-2]
		\arrow[from=1-2, to=1-3]
		\arrow[from=1-3, to=1-4]
		\arrow[from=1-4, to=1-5]
	\end{tikzcd}\]
	from Corollary \ref{cor:SES}, noting that $\Omega(E_{f,n}) \cong E_{f,n}$. As in Remark \ref{rem:Esub}, $E_{f,n}$ does not admit $E_{t}$, $E_{t+1_{k}}$ or $E_{\infty}$ as submodules, so we must have $P = 0$. Then $E_{f,n} \cong N$ is a submodule of a projective-free permutation module, but this contradicts Theorem \ref{thm:pfsub}. We conclude that $\ppdim(E_{f,n}) \neq 1$ and so $\ppdim(E_{f,n}) = 2$ by Proposition \ref{prop:ppdimE}.
\end{proof}

\begin{thm}
	Let $n \geq 3$ and either $f = t$, $f = t+1_{k}$ or $f = \infty$. Then $\ppdim(E_{f,n}) = 2$.
\end{thm}

\begin{proof}
	Suppose that $\ppdim(E_{f,n}) = 1$ and consider the short exact sequence:
	\[\begin{tikzcd}[cramped,sep=small]
		0 & P & {E_{f,n}} & N & 0
		\arrow[from=1-1, to=1-2]
		\arrow[from=1-2, to=1-3]
		\arrow[from=1-3, to=1-4]
		\arrow[from=1-4, to=1-5]
	\end{tikzcd}\]
	from Corollary \ref{cor:SES}, noting that $\Omega(E_{f,n}) \cong E_{f,n}$. We have $P \neq 0$, otherwise $E_{f,n} \cong N$ is a submodule of a projective-free permutation module, contradicting Theorem \ref{thm:pfsub}. Consider the case $f \defeq t$. As in Remark \ref{rem:Esub}, $E_{t,n}$ does not admit $E_{t+1_{k}}$ or $E_{\infty}$ as submodules, so $P \cong E_{t}^{\oplus m}$ for some $m \geq 1$. Note that any submodule of $E_{t,n}$ which is isomorphic to $E_{t}$ will have a $k$-basis of the form $\{\lambda u_{n-1} + \sum_{i=0}^{n-1}(\mu_{i}v_{i}),\, \lambda v_{n-1}\}$ for some $\lambda, \mu_{i} \in k$ with $\lambda \neq 0_{k}$, therefore any such submodule must contain $v_{n-1}$. This means that $E_{t}^{\oplus m}$ cannot be a submodule of $E_{t,n}$ for $m \geq 2$, otherwise the summands would have non-zero intersection. We conclude that $P \cong E_{t}$. An entirely similar argument shows that when $f \defeq t+1_{k}$ or $f \defeq \infty$, we have $P \cong E_{t+1_{k}}$ or $P \cong E_{\infty}$ respectively. By Lemma \ref{lem:cokerE}, in each of these cases we have $N \cong E_{f,n-1}$. This again contradicts Theorem \ref{thm:pfsub}, so we conclude that $\ppdim(E_{f,n}) \neq 1$, hence $\ppdim(E_{f,n}) = 2$ by Proposition \ref{prop:ppdimE}.
\end{proof}

\begin{thm}
	Let $n \geq 3$. Then $\ppdim(W_{2n+1}) = 2$.
\end{thm}

\begin{proof}
	Suppose that $\ppdim(W_{2n+1}) = 1$ and consider the short exact sequence:
	\[\begin{tikzcd}[cramped,sep=small]
		0 & P & {W_{2n+3}} & N & 0
		\arrow[from=1-1, to=1-2]
		\arrow[from=1-2, to=1-3]
		\arrow[from=1-3, to=1-4]
		\arrow[from=1-4, to=1-5]
	\end{tikzcd}\]
	from Corollary \ref{cor:SES}, noting that $\Omega(W_{2n+1}) \cong W_{2n+3}$. Any submodule of $W_{2n+3}$ which is isomorphic to $E_{t}$ will have a $k$-basis of the form $\{\lambda u_{n+1} + \sum_{i=1}^{n+1}(\mu_{i}v_{i}),\, \lambda v_{n+1}\}$ for some $\lambda, \mu_{i} \in k$ with $\lambda \neq 0_{k}$, therefore any such submodule must contain $v_{n+1}$. Similarly, any submodule isomorphic to $E_{t+1_{k}}$ or $E_{\infty}$ must contain $\sum_{i=1}^{n+1}(v_{i})$ or $v_{1}$ respectively. It follows that each of $E_{t}$, $E_{t+1_{k}}$ and $E_{\infty}$ can only occur at most once as a summand of $P$, otherwise the images in $W_{2n+3}$ of repeated summands would have non-zero intersection. Therefore $P$ must be isomorphic to either $0$,\, $E_{t}$,\, $E_{t+1_{k}}$,\, $E_{\infty}$,\, $E_{t} \oplus E_{t+1_{k}}$,\, $E_{t} \oplus E_{\infty}$,\, $E_{t+1_{k}} \oplus E_{\infty}$ or $E_{t} \oplus E_{t+1_{k}} \oplus E_{\infty}$. We know that $W_{2n+3}/E_{t} \cong W_{2n+3}/E_{t+1_{k}} \cong W_{2n+3}/E_{\infty} \cong W_{2n+1}$ regardless of the inclusions, by Lemma \ref{lem:cokerW}. It follows by the isomorphism theorems that $N \cong W_{2n+3}/P$ is isomorphic to either $W_{2n+3}$, $W_{2n+1}$, $W_{2n-1}$ or $W_{2n-3}$. In all cases, this contradicts Theorem \ref{thm:pfsub} since $N$ is a submodule of a projective-free permutation module. We conclude that $\ppdim(W_{2n+1}) \neq 1$ and so $\ppdim(W_{2n+1}) = 2$ by Proposition \ref{prop:ppdimW}.
\end{proof}

The main result, Theorem \ref{thm:main}, now follows.

\begin{rem}
	We saw that the permutation dimension of an arbitrary $kV_{4}$-module is bounded above by the maximum permutation dimension of its indecomposable summands. It is an interesting question as to whether this bound is an equality.
	
	Suppose that none of the summands have permutation dimension equal to two. Then there exists a summand with permutation dimension equal to one if and only if the permutation dimension of the module is equal to one, so in this case the bound is an equality.
	
	The question reduces to whether or not it is possible for a $kV_{4}$-module to have permutation dimension equal to one whilst admitting a summand that has permutation dimension equal to two. This problem remains open.
\end{rem}

\begin{Ack}
	The author would like to thank Martin Gallauer for suggesting such an excellent problem to work on. His insight was crucial to the formulation of several key arguments in this paper.
	
	This work was supported by the Heilbronn Institute for Mathematical Research.
\end{Ack}




\begin{thebibliography}{5}
	
	\bibitem[1]{BG23}
	Paul Balmer and Martin Gallauer.
	\newblock Finite permutation resolutions.
	\newblock \emph{Duke Mathematical Journal}, \mbox{172(2):201--229}, 2023.
	\newblock \href{https://doi.org/10.1215/00127094-2022-0041}{\mbox{DOI: 10.1215/00127094-2022-0041}}.
	
	\bibitem[2]{Bas61}
	V. A. Ba\v{s}ev.
	\newblock Representations of the group $\mathbb{Z}_2 \times \mathbb{Z}_2$
	in a field of characteristic $2$.
	\newblock \emph{Dokl. Akad. Nauk SSSR}, \mbox{141:1015--1018}, 1961.
	
	\bibitem[3]{Ben91}
	D.~J. Benson.  
	\newblock \emph{Representations and Cohomology I}.  
	\newblock Cambridge University Press, 1991.  
	\newblock \href{https://doi.org/10.1017/CBO9780511623615}{\mbox{DOI: 10.1017/CBO9780511623615}}.
	
	\bibitem[4]{CM12}
	Sunil Chebolu and J\'{a}n Min\'{a}c.
	\newblock Representations of the miraculous Klein group.
	\newblock \emph{Ramanujan Mathematical Society Newsletter}, \mbox{22(1):135--145}, 2012.
	\newblock Available \href{https://drive.google.com/file/d/0B_RIasLa4TYoTUFFLWEwODRkZjM4LWVjZjQtNGUzMy05MzQ3LTJlNTc4N2QyNTIzNQ/view}{\mbox{here}} or at \href{https://arxiv.org/abs/1209.4074}{\mbox{arXiv:1209.4074}}.
	
	\bibitem[5]{HR61}
	Alex Heller and Irving Reiner.
	\newblock Indecomposable representations.
	\newblock \emph{Illinois Journal of Mathematics}, \mbox{5:314--323}, 1961.
	
	\bibitem[6]{Wal25}
	Jack Walsh.
	\newblock Permutation dimensions of prime cyclic groups.
	\newblock \emph{Communications in Algebra}, \mbox{54(5):2208--2215}, 2026.
	\newblock \href{https://doi.org/10.1080/00927872.2025.2578215}{\mbox{DOI: 10.1080/00927872.2025.2578215}}.
	
	\bibitem[7]{Web16}
	Peter Webb.  
	\newblock \emph{A Course in Finite Group Representation Theory}.  
	\newblock Cambridge University Press, 2016.  
	\newblock \href{https://doi.org/10.1017/CBO9781316677216}{\mbox{DOI: 10.1017/CBO9781316677216}}.
	
\end{thebibliography}
\end{document}